\documentclass[a4paper,10pt]{article}
\usepackage[english]{babel}
\usepackage[utf8]{inputenc}
\usepackage[T1]{fontenc}
\usepackage{textcomp} 
\usepackage{amsmath}
\usepackage{amsfonts}
\usepackage{amsthm}
\usepackage{mathrsfs}
\usepackage{amssymb}
\usepackage{enumerate}
\usepackage{csquotes}
\usepackage{graphicx}
\usepackage{bbm}
\usepackage[mathscr]{euscript}
\DeclareSymbolFont{rsfs}{U}{rsfs}{m}{n}
\DeclareSymbolFontAlphabet{\mathscrsfs}{rsfs}

\theoremstyle{definition}
\newtheorem{Def}{Definition}[section]

\newtheorem{Rmk}[Def]{Remark}
\newtheorem{Nota}[Def]{Notation} 
\newtheorem{Cond}[Def]{Condition}

\theoremstyle{plain}
\newtheorem{Prop}[Def]{Proposition}
\newtheorem{Thm}[Def]{Theorem}
\newtheorem{Lemma}[Def]{Lemma}
\newtheorem{Cor}[Def]{Corollary}

\newcommand{\R}{\mathbb{R}}
\newcommand{\Q}{\mathbb{Q}}
\newcommand{\E}{\mathbb{E}}

\newcommand{\Z}{\mathbb{Z}}
\newcommand{\N}{\mathbb{N}}

\newcommand{\Osc}{\text{ \rm Osc}}
\newcommand{\ha}{\dim_{\mathcal{H}}}

\DeclareFontFamily{U}{mathx}{}
\DeclareFontShape{U}{mathx}{m}{n}{<-> mathx10}{}
\DeclareSymbolFont{mathx}{U}{mathx}{m}{n}
\DeclareMathAccent{\widehat}{0}{mathx}{"70}
\DeclareMathAccent{\widecheck}{0}{mathx}{"71}

\renewcommand{\epsilon}{\varepsilon}

\def\diam{\mbox{\rm diam}}

\title{Multifractional Hermite processes: definition and first properties}
\author{L. Loosveldt\footnote{University of Luxembourg,Department of Mathematics (DMATH), Maison du Nombre, 6, avenue de la Fonte, L-4364 Esch-sur-Alzette, Grand Duchy Of Luxembourg. \linebreak laurent.loosveldt@uni.lu}}

\begin{document}

\maketitle

\begin{abstract}
We define multifractional Hermite processes which generalize and extend both multifractional Brownian motion and Hermite processes. It is done by substituting the Hurst parameter in the definition of Hermite processes as a multiple Wiener-Itô integral by a Hurst function. Then, we study the pointwise regularity of these processes, their local asymptotic self-similarity and some fractal dimensions of their graph. Our results show that the fundamental properties of multifractional Hermite processes are, as desired, governed by the Hurst function. Complements are given in the second order Wiener chaos, using facts from Malliavin calculus.
\end{abstract}


\section{Introduction}

Fractional Brownian motion with Hurst parameter $h \in (0,1)$ is known to be the unique Gaussian process with $B_h(0)=0$, mean zero and covariance function
\[ \mathbb{E}[B_h(t) B_h(s)]=\frac{c_h}{2} \left( |t|^{2h}+|s|^{2h}-|t-s|^{2h} \right),\]
where $c_h$ is a positive constant only depending on $h$. It was introduced by Kolmogorov, in 1940, to generate Gaussian ``spirals'' in Hilbert spaces \cite{MR0003441}. It is itself a generalization of the famous Brownian motion, when  $h=1/2$, defined by the botanist Robert Brown to describe the movements of pollen grains of the plant Clarkia Pulchella suspended in the water \cite{brown}. The first systematic study of fractional Brownian motion goes back to the famous paper \cite{MR242239} by Mandelbrot and Van Ness, in 1968. Since then, fractional Brownian motion has appeared in many real-life applications in various domains, such as telecommunications, biology, finance, image processing and much more \cite{MR1956041}.

Among its most fundamental properties, fractional Brownian motion has stationary increments and is $h$-self-similar, meaning that, for all $a>0$, the processes $\{a^{-h} B_h(at)\}_{t \in \R}$ and $\{B_h(t) \}_{t \in \R}$ have the same finite-dimensional distributions. The Hurst parameter also rules the regularity of the process since the uniform and pointwise Hölder exponents (see Section \ref{sec:strategy} for a definition) of $B_h$ are almost surely $h$. Actually, it appears that, for some applications, these properties are undesirable. For instance, fractional Brownian motion was used in image synthesis to model artificial mountains \cite{MR952853} but then, the obtained relief has the same (ir)regularity everywhere, which is not realistic. To overcome this drawback, the two papers \cite{plv95} and \cite{MR1462329} introduced independently and from two different perspectives the so-called multifractional Brownian motion. It is defined by substituting the Hurst parameter $h$ by a Hurst function $H(\cdot)$ with values in a compact interval of $(0,1)$. Under some regularity assumptions for the Hurst function, see Conditions \ref{condi} (a), (b) and (c) below, one can show that, almost surely, the function $H$ governed the Hölder regularity of the multifractional Brownian motion. Also, the self-similarity property is turned into a local asymptotic self-similarity property, see Definition \ref{def:selfismilarity} below.

Since the introduction of the multifractional Brownian motion, many authors studied this process, from various perspectives. One can cite for instance the papers \cite{MR2348754,MR2219711} concerning the local time of this process, \cite{MR2188838,MR3597564} for statistical estimation of the Hurst function, \cite{MR2337692,MR3089823} where fractal dimensions are computed, \cite{MR1948227,esserloosveldt2} for studies of the precise pointwise regularity, and \cite{MR3093550,MR3131310,MR3176508} where a stochastic calculus with respect to multifractional Brownian motion is defined. Also, different generalizations has been given such as in \cite{MR1726365,MR1819282,MR1932747}, where a larger class of Hurst functions are considered, in order that the Hölder exponent of the process is, almost surely, of the most general form given in \cite{MR1474098,MR1626706}, or in \cite{MR2177638,MR2351137,MR3023836} where the Hurst function is also random. Finally, various extensions have been given, using larger classes of processes closely related to the fractional Brownian motion like, for instance, the linear multifractional stable motion \cite{MR3293435,MR3640791} or the Surgailis multifractional process \cite{MR2376898,MR4438441}. We also refer to the book \cite{MR3839281} for a very clear view on the known facts about multifractional Brownian motion and related fields. The aim of the current paper is to define an extension of multifractional Brownian motion in an arbitrary Wiener chaos, using the affiliation of fractional Brownian motion in the class of Hermite processes. 

All along this paper, given $d \in \N^*$ and a symmetric function $f \in L^2(\R^d)$, $I_d(f)$ stands for the $d$-multiple Wiener-Itô integral of $f$ with respect to the Brownian motion $\{B(t) \}_{t \in \R}$ defined on a probability space $(\Omega,\mathcal{F},\mathbb{P})$. If $f$ is of the form
\begin{equation}\label{eqn:Wiener-Itôintegral}
f = \sum_{j_1,\ldots,j_d=1}^n a_{j_1,\ldots,j_d} \mathbbm{1}_{[s_{j_1},t_{j_1})} \otimes \cdots \otimes \mathbbm{1}_{[s_{j_d},t_{j_d})},
\end{equation}
where, $\otimes$ stands for the tensor product, $a_{j_1,\ldots,j_d}$ are such that, for all permutation $\sigma$, $a_{\sigma(j_1),\ldots,\sigma_(j_d)}=a_{j_1,\ldots,j_d}$ and $a_{j_1,\ldots,j_d}=0$ as soon as two indices $j_1,\ldots,j_d$ are equal and, for all $1 \leq \ell \neq \ell' \leq d$, $[s_{j_\ell},t_{j_\ell}) \cap [s_{j_{\ell'}},t_{j_{\ell'}}) = \emptyset$, then
\begin{equation}\label{eqn:Wiener-Itôintegral2}
I_d(f) := \sum_{j_1,\ldots,j_d=1}^n a_{j_1,\ldots,j_d} (B(t_{j_1})-B( s_{j_1})) \times \ldots (B(t_{j_d})-B( s_{j_d})).
\end{equation} 
It is straightforward that this last random variable belongs to $L^2(\Omega)$. For a general symmetric $f \in L^2(\R^d)$, $I_d(f)$ is then defined using the density of functions of the form \eqref{eqn:Wiener-Itôintegral} within the set of symmetric square integrable function and by checking that the corresponding random variables \eqref{eqn:Wiener-Itôintegral2} converge in $L^2(\Omega)$. Among many properties that enjoys this integral, we will mainly use the so-called isometry property: for all $f,g$ symmetric function in $L^2(\R^d)$ and $L^2(\R^{d'})$ respectively,
\begin{equation}\label{eqn:isometry}
\mathbb{E} \left[I_d(f) I_{d'}(g) \right]= \begin{cases}
d! \langle f, g \rangle & \text{if $d=d'$} \\
0 & \text{otherwise,}
\end{cases}
\end{equation}
where $\langle \cdot , \cdot \rangle$ stands for the canonical scalar product in $L^2(\R^d)$. The $d$th Wiener chaos is defined as the closed linear subspace of $L^2(\Omega)$ generated by the random variables of the form $I_d(f)$, with $f$ symmetric function in $L^2(\R^d)$.

Now, given $h \in  (\frac{1}{2},1)$, we define, for all $s \geq 0$, the function
\begin{equation}\label{nota:kernel}
 f_h(s, \bullet) \, : \, \R^d \to \R_+ \, : \, \mathbf{x} \mapsto \prod_{\ell=1}^d (s-x_\ell)_+^{\frac{h-1}{d}-\frac{1}{2}}.
\end{equation}
It is easy to show that, for all $t \geq 0$, the function
\[ \int_0^t  f_{h}(s, \bullet) \, ds\]
is symmetric and belongs to $L^2(\R^d)$. Then, the Hermite process of order $d$ and Hurst parameter $h$ is defined as\footnote{Some authors used a normalization constant in the definition of the Hermite process, in order that the process at time $1$ has unit variance.}
\begin{equation}\label{intro:defofhermite}
\left \{ I_d \left(\int_0^t  f_{h}(s, \bullet) \, ds \right) \right\}_{t \in \R_+}.
\end{equation}
When $d=1$, this process reduces to the fractional Brownian motion of Hurst parameter $h$. As soon as $d>1$, the Hermite process of order $d$ is known to be non-Gaussian. Hermite processes first appeared as limit of partial sums of correlated random variables, in the so-called Non-Central Limit Theorem, see \cite{MR550122,MR400329,MR550123}. Apart from Gaussianity, Hermite processes share many properties with fractional Brownian motion such as the stationarity of increments, the $h$-self similarity, the Hölder regularity. These facts are particularly interesting in application where we have to model a phenomena for which the Gaussianity is not a reasonable assumption. See for instance \cite{MR2212691} where the asymptotic distributions in a model for the unit root testing problem, with errors being non-linear transforms of linear processes, are shown to be functionals of Hermite processes.

In this paper we define multifractional Hermite processes by substituting the constant Hurst parameter in \eqref{intro:defofhermite} by a Hurst function with values in a compact interval of $(1/2,1)$. In order to do so, we first introduce the following fields, called ``generators of Multifractional Hermite processes''.

\begin{Def}\label{def:generator}
Given $d \in \N^*$, the \textit{generator of the multifractional Hermite process of order  $d$} is the real-valued centred field $\{X_d(t,h) \}_{(t,h) \in \R_+ \times (\frac{1}{2},1)}$ defined, for all $(t,h) \in \R_+ \times (\frac{1}{2},1)$, by the multiple Wiener-Itô integral
\begin{equation}\label{eqn:def:gene}
X_d(t,h):=I_d \left(\int_0^t  f_{h}(s, \bullet) \, ds \right).
\end{equation}
\end{Def}

Let us remark that, if $h \in (1/2,1)$ is fixed, $\{X_d(t,h) \}_{t \in \R_+}$ is the standard Hermite process of order $d$ and Hurst parameter $h$. In Proposition \ref{prop:holdereggene} below, we show that, for all $d$, there exist a modification of $\{X_d(t,h)\}_{t \in \R_+, h \in (1/2,1)}$ and $\Omega^*$, an event of probability $1$, such that, on $\Omega^*$, the trajectories of this modification are (Hölder) continuous. Then, we identify $\{X_d(t,h)\}_{t \in \R_+, h \in (1/2,1)}$ with this modification and define multifractional Hermite processes as follows.

\begin{Def}\label{def:multihermi}
Given $d \in \N^*$, a compact set $K$ of $(\frac{1}{2},1)$ and a function $H \, : \, \R_+ \to K$, the \textit{multifractional Hermite process of order $d$ and Hurst function $H$} is the process $\{X_d^{H(\cdot)}(t)\}_{t \in \R_+}$ defined, for all $t \in \R_+$, by
\begin{equation}\label{eqn:liengenher}
X_d^{H(\cdot)}(t)=X_d(t,H(t)).
\end{equation}
\end{Def}

\begin{Rmk}
Of course, the trajectories of multifractional Hermite processes and their associated generators depend on the additional parameter $\omega \in \Omega$. In order to ease the notations, all along this paper, when the context is clear, we allow ourself not to explicitly mention this dependence and to write $X_d^{H(\cdot)}(t)$ and $X_d(t,h)$ instead of $X_d^{H(\cdot)}(\omega,t)$ and $X_d(\omega,t,h)$ respectively.
\end{Rmk}

When $d=2$, the multifractional Hermite process corresponds to the multifractional Rosenblatt process previously introduced in the paper \cite{MR2768856}. Nevertheless, Wiener-Itô integrals of order $2$ enjoy specific properties (see the end of Section \ref{sec:strategy} and Section \ref{sec:rosen} in the present paper). Thus, the study undertook here is more general. Moreover, some facts proved in this paper are not considered in \cite{MR2768856}. Among other things, in Section \ref{sec:pointwiseholder} we compute the exact Hölder exponents of the multifractionnal Rosenblatt process (only upper bounds are given in \cite{MR2768856}) and in Section \ref{section:iterated} we establish a law of iterated logarithm. Also, in Section \ref{sec:rosen}, we prove the existence of a continuous and bounded density for increments of the multifractional Rosenblatt process, with the help of Malliavin calculus. It helps us to refine some facts explored in this paper. 

Other multifractional processes in arbitrary Wiener chaoses have already been defined in the literature. In the paper \cite{MR3824677}, the authors consider a multifractional generalization of processes introduced in \cite{MR3192502}. They are defined with an alternative kernel which facilitates the computation of a wavelet-type expansion. Also, in the papers \cite{MR3102483,marty}, the author obtains some multifractional processes in arbitrary Wiener chaoses as limits of weighted sums of multifractional Gaussian fields. These processes are a priori not directly related to the ones defined in this paper, as it is already the case in the first order chaos, see \cite{MR2451054}.

With this paper, we hope to open the door to further investigations concerning multifractional Hermite processes. We believe that many interesting research questions could be addressed, similarly to what have been done with the multifractional Brownian motion. Also, we think that multifractional Hermite processes could be used in applications to model phenomena where both Gaussianity and constant regularity can not be assumed. To motivate the introduction of multifractional Hermite processes, we focus on some first important properties concerning the behaviour of stochastic processes: Hölder regularity, the law of iterated logarithm, local asymptotic self-similarity and fractal dimensions for the graph. These notions are defined in Section \ref{sec:strategy} as well as the main strategies used to state and prove our main theorems. Section \ref{sect:modulus} is mainly concerned in giving an uniform modulus of continuity for multifractional Hermite processes. In section \ref{sec:pointwiseholder}, we provide a lower bound for the oscillations of multifractional Hermite processes. Section \ref{section:iterated} is devoted to prove a law of iterated logarithm. In Section \ref{sec:locaasym}, the local asymptotic self-similarity is discussed. Section \ref{sec:dim} deals with estimates for the Hausdorff and box-counting dimensions of the graph of multifractional Hermite processes. Finally, in Section \ref{sec:rosen}, some complements concerning the fractal dimensions of the graph of the multifractional Rosenblatt process are given, using specific arguments from the Wiener chaos of order $2$ and  Malliavin calculus. 

Our results show that, as desired, fundamental properties of multifractional Hermite processes are governed by their associated Hurst function.

\section{Preliminaries, strategy and main results}\label{sec:strategy}

As stated in the Introduction, the definition of a multifractional Hermite process relies on a modification of its generator which is almost surely Hölder continuous. Let us start by recalling the definition of this notion.

\begin{Def}
If $f$ a is (deterministic) continuous function defined on a interval $I$ of $\R$, the \textit{oscillation} of $f$ on $I$ is defined by
\begin{equation}\label{eqn:oscillation}
\Osc(f,I):= \sup_{t,s \in I} |f(t)-f(s)|.
\end{equation}
We say that $f$ belongs to the \textit{pointwise Hölder space} at $t_0 \in I$ and of order $\alpha \in (0,1)$ if there exist $R>0$ and $C>0$ such that, for all $0 < r < R$,
\begin{equation}\label{eqn:defholder}
\Osc(f,[t_0-r,t_0+r] \cap I) \leq C r^{\alpha}.
\end{equation}
In this case, we note $f \in C^\alpha (t_0)$. It is easy to check that, if $\alpha < \beta$, then $C^\beta(t_0) \subseteq C^\alpha(t_0)$. Therefore, the pointwise Hölder exponent of $f$ at $t_0$ is defined as
\[ h_f(t_0) := \sup \{ \alpha \in (0,1) \, : \, f \in C^\alpha (t_0) \}.\]
If, for all $t_0 \in I$, $f \in C^\alpha (t_0)$, with an uniform constant $C>0$ in \eqref{eqn:defholder}, we say that $f$ is \textit{uniformly Hölder} on $I$ of order $\alpha$ and we note $f \in C^\alpha(I)$. The uniform Hölder exponent of $f$ on $I$ is then naturally defined as
\[ H_f(I) := \sup \{ \alpha \in (0,1) \, : \, f \in C^\alpha (I) \}.\]
Of course, for all $t_0 \in I$, we have $h_f(t_0) \geq H_f(I)$.
\end{Def}

One of the easiest and most standard way to provide information concerning the Hölder regularity of a stochastic process is to use Kolmogorov continuity theorem. On this purpose, one has to obtain bounds for the norms in $L^p(\Omega)$ of the increments of the process. It is precisely the aim of the next proposition. In fact, this result will be crucial in numerous occasions all along this paper.

\begin{Prop}\label{prop:holdofgene}
Let $d \in \N^*$, $K$ be a compact set of $(\frac{1}{2},1)$ and $I$ be a compact interval of $\R_+$. There exist a positive deterministic constant $c_1$ only depending on $d$ an $K$ and a positive deterministic constant $c_2$, only depending on $d$, $K$ and $I$, such that, for all $t,u \in I$ and $h_1,h_2 \in K$,
\[  \| X_d(t,h_1)-X_d(u,h_2)\|_{L^2(\Omega)} \]
is bounded from above by $c_1 |t-u|^{\min \{h_1,h_2 \}} + c_2|h_1-h_2| $ and from below by $ c_1|t-u|^{\min \{h_1,h_2 \}} -c_2|h_1-h_2|$.
\end{Prop}

\begin{proof}
Without loss of generality, one can assume $K=[a,b]$ with $\frac{1}{2}<a<b<1$ and $h_1 < h_2$. For all $t,u \in I$ and $h_1,h_2 \in K$, of course, we have
\begin{align}\label{eqn:intnormel2}
\|X_d(t,h_1)-X_d(u,h_1)\|_{L^2(\Omega)} &- \|X_d(u,h_1)-X_d(u,h_2)\|_{L^2(\Omega)}\leq  \nonumber \\  \|X_d(t,h_1) &-X_d(u,h_2) \|_{L^2(\Omega)} \nonumber \\ \leq \|X_d(t,h_1)-X_d(u,h_1)\|_{L^2(\Omega)}&+ \|X_d(u,h_1)-X_d(u,h_2)\|_{L^2(\Omega)}.
\end{align}

Using the self-similarity and stationarity of increments of Hermite processes, we know that there exists a deterministic constant $c_1>0$, only depending on $d$ and $K$, such that
\[\|X_d(t,h_1)-X_d(u,h_1)\|_{L^2(\Omega)} =  c_d|t-u|^{h_1}. \]

Thus, it only remains to bound $\|X_d(u,h_1)-X_d(u,h_2)\|_{L^2(\Omega)}$.
Using the isometry property \eqref{eqn:isometry} for Wiener-Itô integrals and Definition \ref{def:generator}, we know that, recalling the notation \eqref{nota:kernel} and writing $d\mathbf{x}$ for $dx_1\ldots dx_d$,
\begin{align}\label{eqn:intnormel2bis}
& \|X_d(u,h_1)-X_d(u,h_2)\|_{L^2(\Omega)} \nonumber \\
& \quad = d! \left( \int_{\R^d} \left( \int_0^u f_{h_1}(s, \mathbf{x})-f_{h_2}(s, \mathbf{x}) \, ds \right)^2 d\mathbf{x}\right)^\frac{1}{2}.
\end{align}
For all $(s,\mathbf{x})$ with $\prod_{\ell=1}^d (s-x_\ell)_+>0$  fixed, by mean value theorem, there is $h' \in [h_1,h_2]$ such that
\[ |f_{h_1}(s, \mathbf{x})-f_{h_2}(s, \mathbf{x})| =|h_1-h_2| f_{h'}(s, \mathbf{x})  \left|\ln \left(\prod_{\ell=1}^d (s-x_\ell)_+\right) \right|. \]
Now, using the fact that, for all $\varepsilon >0$,
\[ \lim_{x \to 0^+} x^\varepsilon \log(x^{-1})=0^+ \quad \text{et} \quad \lim_{x \to + \infty} \frac{\log(x)}{x^\varepsilon}=0^+,\]
one can choose $\varepsilon>0$ such that $\frac{1}{2}<a-\varepsilon<b+\varepsilon<1$ and find a deterministic constant $c_\varepsilon>0$ for which, for all $\mathbf{x} \in \R^d$,
\begin{equation}\label{eq:meanvalfh}
\left|\int_0^u f_{h_1}(s, \mathbf{x})-f_{h_2}(s, \mathbf{x}) \, ds \right| \leq c_\varepsilon |h_1-h_2| \int_0^u f_{a- \varepsilon}(s, \mathbf{x}) +f_{b+ \varepsilon}(s, \mathbf{x}) \, ds
\end{equation}
Plugging this into \eqref{eqn:intnormel2bis} and using again the isometry property \eqref{eqn:isometry} for Wiener-Itô integrals, we get
\begin{align*}
&\|X_d(u,h_1)-X_d(u,h_2)\|_{L^2(\Omega)} \\
 & \quad \leq c_\varepsilon |h_1-h_2| \left( \|X_d(u,a-\varepsilon)\|_{L^2(\Omega)} + \|X_d(u,b+\varepsilon)\|_{L^2(\Omega)} \right) \\
& \quad\leq c_d c_\varepsilon |h_1-h_2| (|u|^{a-\varepsilon} + |u|^{b+\varepsilon}) \\
& \quad \leq c_2 |h_1-h_2|,
\end{align*}
for a positive deterministic constant $c_2>0$ only depending on $d$, $K$ and $I$
\end{proof}

The next corollary is then a direct consequence of the hypercontractivity property on Wiener chaoses, see \cite[Theorem 2.7.2]{MR2962301}.
\begin{Cor}\label{cor:lpnorm}
Given $d \in \N^*$ and $K$ a compact set of $(\frac{1}{2},1)$, let $I$ be a compact interval of $\R_+$. For any $p \geq 1$ there exists a positive deterministic constant $c_p$, only depending on $d$, $p$, $K$ and $I$, such that, for all $t,u \in I$ and $h_1,h_2 \in K$,
\begin{equation}\label{momentp}
\| X_d(t,h_1)-X_d(u,h_2)\|_{L^p(\Omega)} \leq c_p \left(|t-u|^{\min \{h_1,h_2 \}} + |h_1-h_2| \right).
\end{equation}
\end{Cor}

Inequality \eqref{momentp} combined with Kolmogorov continuity theorem are enough to consider the Hölder regularity of generators of multifractional Hermite processes.

\begin{Prop}\label{prop:holdereggene}
Given $d \in \N^*$, there exist a modification of the field \linebreak $\{X_d(t,h)\}_{(t,h) \in \R_+ \times (1/2,1)}$, also denoted by $\{X_d(t,h)\}_{(t,h) \in \R_+ \times (1/2,1)}$, and $\Omega^*$, an event of probability $1$, such that, on $\Omega^*$, given $I$, a compact interval of $\R_+$, and $K$, a compact set of $(\frac{1}{2},1)$, for all $0<a < \inf K$, there exists a finite positive random variable $C$ such that, for all $t,u \in I$ and $h_1,h_2 \in K$,
\begin{equation}\label{ineg:holdereggene}
|X_d(t,h_1)-X_d(u,h_2)| \leq C (|t-u|+|h_1-h_2|)^a.
\end{equation}
\end{Prop}

\begin{proof}
Using \eqref{momentp}, we see that, for all $p>0$ there exists a deterministic constant $c>0$, only depending on$d$, $p$, $I$ and $K$, such that, for all $t,u \in I$ and $h_1,h_2 \in K$,
\[ \| X_d(t,h_1)-X_d(u,h_2)\|_{L^p(\Omega)} \leq c (|t-u|+|h_1-h_2|)^{\inf K} \]
and the conclusion follows by applying a strong version of Kolmogorov continuity theorem, see for instance \cite[Theorem 2.5.1 pages 165 and 166]{MR1914748}.
\end{proof}

Starting from now, we identify the generators of the multifractional Hermite processes with their continuous modification. Once this identification made, the multifractional Hermite process, of order $d$ and Hurst function $H$, $\{X_d^{H(\cdot)}(t) \}_{t \in \R_+}$ is defined by the equality \eqref{eqn:liengenher}. Let us now present our main results which focus on fundamental properties of these processes.

Hölder regularity provide nice information about the pointwise and global behaviour of the functions we consider. Nevertheless, often we are interested by more precise bound for the oscillations. It can be done by the mean of moduli of continuity.
\begin{Def}
If $f$ a is (deterministic) function defined on a interval $I$ of $\R$, we say that a continuous increasing function $\rho$ defined on $\R_+$ and such that $\lim_{r \to 0^+} \rho(r)=0$ is a \textit{modulus of continuity} for $f$ at $t_0 \in I$ if 
\begin{equation}\label{eqn:modulus}
 \limsup_{r \to 0^+}\frac{\Osc(f,[t_0-r,t_0+r] \cap I)}{\rho(r)} < + \infty.
\end{equation}
Moreover, if
\[ \limsup_{r \to 0^+}\frac{ \sup_{t_0 \in I}\Osc(f,[t_0-r,t_0+r] )}{\rho(r)} < + \infty\]
we say that $\rho$ is an \textit{uniform modulus of continuity} for $f$ on $I$.
\end{Def}
\begin{Rmk}
Of course, if $f$ is $\alpha$-Hölder, the function $r \mapsto r^\alpha$ is a modulus of continuity for $f$. Hölder regularity only compare the oscillations with power functions while, with moduli of continuity, one can deduce more precise and relevant information concerning the analysed function. It is particularly true when we consider stochastic processes, see for instance \cite{esserloosveldt,dawloosveldt,esserloosveldt2}. Note that one can define generalized Hölder spaces associated with modulus of continuity \cite{MR3002607,MR3812831,MR4019664} and that these spaces lead to specific multifractal formalisms \cite{MR4304490,MR4447306}.
\end{Rmk}
While considering the multifractional Hermite process $X_d^{H(\cdot)}$, we say that $h$ (resp. $\rho$) is a (pointwise or uniform) Hölder exponent (resp. modulus of continuity) for $X_d^{H(\cdot)}$ if it is a Hölder exponent (resp. modulus of continuity) for all the sample paths $t \mapsto X_d^{H(\cdot)}(t)$ on an event of probability $1$.

\begin{Nota}
Given $d \in \N^*$, a compact set $K$ of $(\frac{1}{2},1)$ and a function \linebreak $H \, : \, \R_+ \to K$, if $I$ is a compact interval of $\R_+$, we note
\[ \underline{H}(I) := \min \{H(I)\} \quad \text{and} \quad \overline{H}(I) := \max \{H(I)\}.  \]
\end{Nota}

While studying multifractional processes, authors generally require a regularity assumption for the function $H$ in order to consider the regularity of the process itself, see for instance \cite{MR3839281,MR1462329,plv95}. Here, we will also work with such conditions.

\begin{Cond}\label{condi}
Given $d \in \N^*$ and a compact set $K$ of $(\frac{1}{2},1)$, we say that the Hurst function $H \, : \, \R_+ \to K$ satisfies
\begin{enumerate}[(a)]
\item the uniform $\min$-Hölder regularity condition if, for all compact interval $I$ of $\R_+$, there exists $\gamma \in (\underline{H}(I) ,1)$ such that $H \in C^\gamma (I)$;
\item the pointwise Hölder  condition if, for all $t \in \R_+$,  there exists $\gamma \in (H(t) ,1)$ such that $H \in C^\gamma (t)$;
\item the local  Hölder  condition if, for all $t \in \R_+$,  there exist a compact interval $I_t \subset \R_+$ and $\gamma \in (H(t),1)$ such that $t \in I_t$ and $H \in C^\gamma (I_t)$;
\end{enumerate}
\end{Cond}
All along this paper, to be as general as possible, we use alternatively Condition \ref{condi} (a), (b) or (c) to state and prove our results. Note that if the Hurst function $H \, : \, \R_+ \to K$ is such that, for all compact interval $I$ of $\R_+$, there exists $\gamma \in (\overline{H}(I) ,1)$ for which $H \in C^\gamma (I)$, then Conditions \ref{condi} (a), (b) and (c) are obviously satisfied.

Our first main result consists in providing, under Condition \ref{condi} (a), an uniform modulus of continuity for each multifractional Hermite processes. In Section \ref{sect:modulus}, we prove the following Theorem.

\begin{Thm}\label{thm:modulusofcontinuity}
Given $d \in \N^*$, a compact set $K$ of $(\frac{1}{2},1)$ and a Hurst function $H \, : \, \R_+ \to K$ satisfying Condition \ref{condi} (a), there exists $\Omega^*_1$, an event of probability $1$, such that, on $\Omega^*_1$, for all compact interval $I$ of $\R_+$
\[  \limsup_{r \to 0^+}\frac{ \sup_{t_0 \in I} \Osc(X_d^{H(\cdot)},[t_0-r,t_0+r] \cap I )}{r^{\underline{H}(I)} (\log r^{-1})^\frac{d}{2}} < + \infty. \]
\end{Thm}

Under Condition \ref{condi} (b), one can compute the pointwise Hölder regularity of the process. As desired, it is governed by the Hurst function, as stated in our second main Theorem, proved in Section \ref{sec:pointwiseholder}.

\begin{Thm}\label{thm:holderregularity}
Given $d \in \N^*$, a compact set $K$ of $(\frac{1}{2},1)$ and a Hurst function $H \, : \, \R_+ \to K$ satisfying Condition \ref{condi} (b), there exists $\Omega^*_2$, an event of probability $1$, such that on $\Omega^*_2$, for all $t_0 \in \R_+$, we have
\[ h_{X_d^{H(\cdot)}}(t_0)=H(t_0).\]
\end{Thm}

In fact, Theorem \ref{thm:holderregularity} is a consequence of the stronger Theorem \ref{thm:lowerbound} below which gives a lower bound for the oscillations of multifractional Hermite processes.

When we study the pointwise regularity of a stochastic process, we are often interested in providing a so-called law of iterated logarithm. It shows that, almost surely, the oscillations at most of the points (in the sense of Lebesgue measure) can be bounded from below and above by a modulus of continuity featuring an iterated logarithm. In Section \ref{section:iterated}, we show that multifractional Hermite processes enjoy this property.

\begin{Thm}\label{thm:iterated}
Given $d \in \N^*$, a compact set $K$ of $(\frac{1}{2},1)$ and a Hurst function $H \, : \, \R_+ \to K$ satisfying Condition \ref{condi} (c), there exists $\overline{\Omega}$, an event of probability $1$, such that on $\overline{\Omega}$, for (Lebesgue) almost every $t_0 \in \R_+$, we have
\begin{equation}\label{eq:iterated}
 0< \limsup_{r \to 0^+} \frac{ \Osc(X_d^{H(\cdot)},[t_0-r,t_0+r] \cap \R_+ )}{r^{H(t)} (\log(\log r^{-1}))^\frac{d}{2}} < \infty.
\end{equation}
\end{Thm}

As shown by Theorem \ref{thm:holderregularity}, if the function $H$ is non constant, almost surely, the pointwise Hölder exponent of the multifractional Hermite process $X_d^{H(\cdot)}$ moves from one point to another. In particular, there is no hope that it is a self-similar process, see \cite[Proposition 1.60]{MR3839281}. For this reason, one prefers to check a weaker assumption, the so-called local asymptotic self-similarity.

\begin{Def}\label{def:selfismilarity}
A real-valued stochastic process $\{X(t) \}_{t \in \R_+}$ is \textit{weakly locally asymptotically self-similar} of order $h>0$ at the point $t_0$ with tangent process $\{Y(t)\}_{t \geq 0}$ if the sequence of process $\{ \varepsilon^{-h} (X(t_0+ \varepsilon t)-X(t_0))\}_{t \in \R_+}$ converges to the process $\{Y(t) \}_{t \in \R_+}$ in finite dimensional distributions, as $\varepsilon \to 0^+$. When $\{X(t)  \}_{t \in \R_+}$ and $\{Y(t) \}_{t \in \R_+}$ has, almost surely, continuous path and if the previous convergence also holds in the sense of continuous function over an arbitrary compact set of $\R_+$, we say that $\{X(t)  \}_{t \in \R_+}$ is \textit{strongly locally asymptotically self-similar} of order $h>0$ at the point $t_0$, with tangent process $\{Y(t)\}_{t \in \R_+}$.
\end{Def}

Of course, the strong local asymptotic self-similarity implies the weak local asymptotic self-similarity. Conversely, let us assume $\{X(t)  \}_{t \in \R_+}$ is weakly locally asymptotically self-similar of order $h>0$ at the point $t_0$ with tangent process $\{Y(t) \}_{t \in \R_+}$. Let $a>0$ be an arbitrary fixed real-number and, for $\varepsilon>0$, let $\mathbb{P}_\varepsilon^a$ be the probability measure induced by $\{ \varepsilon^{-h} (X(t_0+ \varepsilon t)-X(t_0))\}_{t \in \R_+}$ on the Borel $\sigma$-algebra of\footnote{As usual, $C([0,a],\R)$ stands for the set of real continuous function on $[0,a]$. } $C([0,a],\R)$. In order, to show that the convergence holds in the strong sense, it suffices to show, since $a>0$ is arbitrary, that the family $(\mathbb{P}_\varepsilon^a)_{\varepsilon >0}$ is relatively compact. Using Prohorov's criterion (see \cite[Section 5 in Chapter 1]{MR1700749} for a comprehensive view), it reduces to show that, for all $\delta>0$
\begin{equation}\label{eq:prohorov}
\lim_{\eta \to 0^+} \limsup_{\varepsilon \to 0^+} \mathbb{P} \left( \sup_{s,t \in [0,a], |t-s| \leq \eta} \left| \frac{X(t_0+\varepsilon t)-X(t_0+\varepsilon s)}{\varepsilon^h} \right| \geq \delta \right) =0.
\end{equation}

In Section \ref{sec:locaasym}, we use this technique to prove the local asymptotic self-similarity of the multifractional Hermite process.

\begin{Thm}\label{thm:localasym}
Let $d \in \N^*$, $K$ be a compact set of $(\frac{1}{2},1)$ and $H \, : \, \R_+ \to K$ be a Hurst function. If $H$ satisfies Condition \ref{condi} (b) then, for all $t_0 \geq 0$, the multifractional Hermite process $\{X_d^{H(\cdot)}(t) \, : \, t \geq 0 \}$ is weakly locally asymptotically self-similar of order $H(t_0)$ at $t_0$ with tangent process $\{X_d(t,H(t_0)) \, : \, t \geq 0 \}$, the Hermite process of order $d$ and Hurst parameter $H(t_0)$. Moreover, if $H$ satisfies Condition \ref{condi} (c), then this property also holds in the strong sense. 
\end{Thm}

The last notions that we consider in this paper to study the behaviour of a given multifractional Hermite process are the Hausdorff and box-counting dimensions of its graph. We refer to the fundamental book \cite{MR3236784} for details and proofs concerning these quantities.

\begin{Def}
Given $d \in \N^*$, a set $A \subseteq \R^d$ and $\varepsilon,h>0$, the quantity
\[ \mathcal{H}_\varepsilon^h(A) := \inf  \{ \sum_{j} \diam^h(A_j)  :  A \subseteq \bigcup_j A_j \text{ and, } \forall j, \diam(A_j) < \varepsilon \}\]
where, as usual, $\diam$ stands for the diameter, is called the \textit{$(h,\varepsilon)$-Hausdorff outer measure} of $A$. Moreover, for all $h>0$, the application $\varepsilon \mapsto \mathcal{H}_\varepsilon^h (A)$ is decreasing and it follows that the \textit{$h$-dimensional Hausdorff outer measure}
\[ \mathcal{H}^h (A) := \lim_{\varepsilon \to 0^+} \mathcal{H}_\varepsilon^h(A) \]
is well-defined. 
\end{Def}

The crucial property of Hausdorff outer measures is that, for any non-empty set $A$, there exists a critical value $h_0$ such that
\[\mathcal{H}^h(A)= \infty ~ \forall \, h <h_0 \quad \text{and} \quad \mathcal{H}^h(A)= 0 ~ \forall \, h >h_0 .\]

\begin{Def}
Given $d \in \N^*$ and a non-empty set $A \subseteq \R^d$, \textit{the Hausdorff dimension of $A$} is 
\[ \ha (A) = \sup \{h>0 : \mathcal{H}^h(A) = \infty \} = \inf \{h>0 : \mathcal{H}^h(A)=0 \},\]
while, by convention, $\ha(\emptyset)=- \infty$. 
\end{Def}

An alternative notion of dimensions for fractal sets are given by the box-counting dimensions.

\begin{Def}
Given $d \in \N^*$, a non-empty bounded set $A \subseteq \R^d$ and $\varepsilon>0$, let $N_\varepsilon(A)$ be the smallest number of sets of diameter at most $\varepsilon$ which can cover $A$. The quantities
\[ \underline{\dim}_{\mathcal{B}}(A) := \liminf_{\varepsilon \to 0^+} \frac{\log(N_\varepsilon(A)) }{- \log( \varepsilon)} \text{ and } \overline{\dim}_{\mathcal{B}}(A) := \limsup_{\varepsilon \to 0^+} \frac{\log(N_\varepsilon(A)) }{- \log( \varepsilon)}\]
are, respectively, the \textit{lower and upper box-counting dimensions} of $A$. If they are equal, the common value is refereed as the \textit{box-counting dimension} of $A$ and we denote it $\dim_{\mathcal{B}}(A)$.
\end{Def}

We also refer to \cite{MR3236784} for all the properties of these dimensions and a clear presentation of their respective utilities and interpretations. Here, we will mainly use the fact that, for any non-empty bounded set $A \subseteq \R^d$,
\begin{equation}\label{ine:dims}
\ha(A) \leq \underline{\dim}_{\mathcal{B}}(A) \leq \overline{\dim}_{\mathcal{B}}(A).
\end{equation}
Also, we use the fact that, for all $A,B$ subset of $\R^d$,
\begin{equation}\label{eqn:incredimension}
(A \subseteq B) \Rightarrow \ha(A) \leq \ha(B) 
\end{equation}

In this paper, given $d \in \N^*$, a compact set $K$ of $(\frac{1}{2},1)$, a Hurst function $H \, : \, \R_+ \to K$ and a compact interval $I \subset \R_+$, we are interested in the dimensions of the graph
\[ \mathcal{G}_d(I):=\{ (t, X_d^{H(\cdot)}(t)) \, : \, t \in I\}  \]
In view of inequalities \eqref{ine:dims}, our strategy consists in bounding from above the (upper) box-counting dimension and from below the Hausdorff dimension. The main Theorem of Section \ref{sec:dim} can then be stated as follows.

\begin{Thm}\label{thm:dim}
Given $d \in \N^*$, a compact set $K$ of $(\frac{1}{2},1)$, a Hurst function $H \, : \, \R_+ \to K$ satisfying Condition \ref{condi} (a) and a compact interval $I \subset \R_+$, there exists $\widetilde{\Omega}$, an event of probability $1$, such that on $\widetilde{\Omega}$, we have
\[ 1+ \frac{1-\underline{H}(I)}{d} \leq \ha \left( \mathcal{G}_d(I)\right) \leq \overline{\dim}_{\mathcal{B}}\left( \mathcal{G}_d(I)\right) \leq 2-\underline{H}(I).\]
\end{Thm}

When $d=1$, inequalities in Theorem \ref{thm:dim} are equalities and we recover the well-known result of \cite{plv95}. Unfortunately, for $d>1$, we have a disparity between the lower and upper bounds for the fractal dimensions. It comes from the estimates that can be made on the probabilities
\begin{equation}\label{eqn:probaintro}
\mathbb{P}(|X_d^{H(\cdot)} (t)-X_d^{H(\cdot)} (u)| \leq x),
\end{equation}
for $t,u ,x \geq 0$, see Proposition \ref{lemma:boundproba} below. It is unknown whether a general (multifractional) Hermite process admits a continuous and bounded density and thus we have to estimate \eqref{eqn:probaintro} with the so-called Carbery-Wright inequality, Lemma \ref{lemma:carbery} in this paper, which induces this factor $\frac{1}{d}$. Nevertheless, for $d=2$, one can use specific arguments from the second order Wiener chaos, \cite[Section 2.4]{MR2962301}. Indeed, if $f$ is a symmetric function in $L^2(\R^2)$, let us consider the Hilbert-Schmidt operator defined as
\[ \mathcal{A}_f \, : \, L^2(\R) \to L^2(\R) \, : \, g \mapsto \int_{\R} f( \cdot,y) g(y) \, dy. \]
Then, let $\{\lambda_{f,j} \}_{j \in \N}$ and $\{e_{f,j}\}_{j \in \N}$ indicate, respectively, the eigenvalues of $\mathcal{A}_f$ and the corresponding eigenvectors. The system $\{e_{f,j} \}_{j \in \N}$ is orthonormal in $L^2(\R)$, the sequence $\{\lambda_{f,j} \}_{j \in \N}$ belongs to $\ell^p$, for all $p \geq 2$, and $f$ has the expansion
\[
f = \sum_{j \in \N} \lambda_{f,j}  e_{f,j} \otimes e_{f,j},
\]
with convergence in $L^2(\R^2)$. In particular, from this last equality, one can write
\begin{equation}\label{expansionchaos2}
I_2(f) = \sum_{j \in \N} \lambda_{f,j} \left(I_1(e_{f,j})^2-1 \right),
\end{equation}
with convergence in $L^2(\Omega)$. Let us also note that the orthonormality of $\{e_{f,j} \}_{j \in \N}$ entails
\begin{equation}\label{eq:norenvp}
\|f\|^2_{L^2(\R^2} =\sum_{j \in \N} \lambda_{f,j}^2.
\end{equation}

In Section \ref{sec:rosen}, we take advantage of this expansion, together with arguments from Malliavin calculus and the paper \cite{MR3132731}, to prove the following improvement of Theorem \ref{thm:dim} in the second order Wiener chaos. We recall that, in this case, the multifractional Hermite process corresponds to the multifractional Rosenblatt process.

\begin{Thm}\label{thm:rose}
Given a compact set $K$ of $(\frac{1}{2},1)$, a Hurst function $H \, : \, \R_+ \to K$ satisfying Condition \ref{condi} (a) and a compact interval $I \subset \R_+$, there exists $\widetilde{\Omega}_2$, an event of probability $1$, such that on $\widetilde{\Omega}_2$, we have
\[  \ha \left( \mathcal{G}_2(I)\right) = \dim_{\mathcal{B}}\left( \mathcal{G}_2(I)\right) = 2-\underline{H}(I).\]
\end{Thm}

Note that this disparity of results also appeared in the standard case, where the Hurst function is constant, see \cite{MR3192502}. We conjecture that, in fact, the equality holds for any (multifractional) Hermite processes. A strategy to prove this fact would be to show that (multifractional) Hermite processes admit continuous and bounded densities. It is still an open question which goes far beyond the scope of this paper.

\section{Uniform modulus of continuity}\label{sect:modulus}

Let us now focus on the continuity and regularity of multifractional Hermite processes. Let us first remark that, on the event $\Omega^*$ induced by Proposition \ref{prop:holdereggene}, $X_d^{H(\cdot)}$ is always continuous at $0$. Indeed, if $\omega \in \Omega^*$ is fixed and $(t_j)_j$ is a sequence which converges to $0$, then let us consider a subsequence $(t_{k(j)})_j$ of $(t_j)_j$. As $H$ has a compact image, there is a subsequence $(t_{l(k(j)}))_j$ such that $H(t_{l(k(j))}) \to \widetilde{H}_0$, for some $\widetilde{H}_0 \in K$. Then, inequality \eqref{ineg:holdereggene} entails
\[X_d^{H(\cdot)}(\omega,t_{l(k(j))})=X_d(\omega,t_{l(k(j))},H(t_{l(k(j))})) \to X_d(\omega,0,\widetilde{H}_0)=0=X_d^{H(\cdot)}(\omega,0). \]
Thus, any subsequence of $(X_d^{H(\cdot)}(\omega,t_{j}))_j$ has a subsequence which converges to $0$, which means that $(X_d^{H(\cdot)}(\omega,t_{j}))_j$ also converge to $0$.

Of course, if $H$ is a continuous function, \eqref{eqn:liengenher} and \eqref{ineg:holdereggene} imply that, on $\Omega^*$, $X_d^{H(\cdot)}$ is continuous on $\R_+$. At the opposite, if $H$ is discontinuous at a point $t_0 \neq 0$, using again the fact that the image of $H$ is compact, we know that there exists $(t_j)_j$ such that $t_j \to t_0$ and $t_j \to H_0 \neq H(t_0)$. Then, from the isometry property \eqref{eqn:isometry} for Wiener-Itô integrals, we get
\begin{align*}
& \|X_d(t_0,H_0)-X_d(t_0,H(t_0))\|_{L^2(\Omega)} \nonumber \\
& \, = d! \left( \int_{\R^d} \left( \int_0^{t_0} f_{H_0}(s, \mathbf{x})-f_{H(t_0)}(s, \mathbf{x}) \, ds \right)^2 d\mathbf{x}\right)^\frac{1}{2}.\\
& \, = d! \left( \int_{\R^d} \left( \int_0^{t_0} \prod_{\ell=1}^d(s-x_\ell)_+^{\frac{H(t_0)-1}{d}-\frac{1}{2}} \left( \prod_{\ell=1}^d(s-x_\ell)_+^{\frac{H_0-H(t_0)}{d}}-1\right) \, ds \right)^2 d\mathbf{x}\right)^\frac{1}{2} \\
& \,>0.
\end{align*}
It means that one can find an event $\Omega_{t_0}$ of probability $1$ such that, for all $\omega \in \Omega_{t_0}$, $X_d^{H(\cdot)}$ is discontinuous at $t_0$.

From this discussion, we see that in order to insure the almost sure continuity of multifractional Hermite processes, we have to assume that the Hurst function is continuous. 

Under Condition \ref{condi} (a), by Proposition \ref{prop:holdereggene}, one easily see that, almost surely, for all compact interval $I$ of $\R_+$,  $X_d^{H(\cdot)}$ is Hölder continuous on $I$, with Hölder exponent at least $\underline{H}(I)$. 

Here, we aim at giving a more precise result by providing an almost sure uniform modulus of continuity for $X_d^{H(\cdot)}$. On this purpose, let us recall the following important fact, see for instance \cite[Theorem 6.7]{MR1474726}.

\begin{Lemma}\label{thm:jason}
For all $d \geq 1$, there exists an universal deterministic constant $c_d>0$ such that, for any random variable $X$ in the Wiener chaos of order $d$,  and $y \geq 2$,
\[ \mathbb{P}(|X| \geq y \|X \|_{L^2(\Omega)}) \leq \exp(-c_d y^\frac{2}{d}).\]
\end{Lemma}

Let us introduce some notations. For all $j \in \N$ and $t \in \R_+$, $k_j^-(t)$ is the unique positive integer such that $t \in [k_j^{-}(t)2^{-j},(k_j^-(t)+1)2^{-j})$ and we set $k_j^+(t):=k_j^-(t)+1$. Let us remark that, for all $t \in \R_+$, $k_j^-(t), k_j^+(t) \to t$ as $j \to + \infty$. Also, note that, for all $j \in \N$, 
\begin{align*}
 \{k_{j+1}^-(t),k_{j+1}^+(t)\} &\subset \{2k_j^-(t),2k_j^-(t)+1,2k_j^-(t)+2\} \\ & =\{2k_j^+(t),2k_j^+(t)-1,2k_j^+(t)-2\}. 
\end{align*}

\begin{proof}[Proof of Theorem \ref{thm:modulusofcontinuity}]
For all $(j,k) \in \N^2$, we write $X_{j,k}:=X_d^{H(\cdot)}(k2^{-j})$ Let us fix $n \in \N$ and, if $c_d$ is the constant given by Lemma \ref{thm:jason}, $c> \left(\frac{\ln(2)}{c_d} \right)^{\frac{d}{2}}$. For all $j \in \N$, let us consider the event $A_j$ defined by
\[\left(\exists 0 \leq k \leq n 2^j, k' \in \, \{2k,2k \pm 1,2k \pm 2\}  :  \frac{|X_{j+1,k'}-X_{j,k}|}{\|X_{j+1,k'}-X_{j,k}\|_{L^2(\Omega)} } \geq  c j^\frac{d}{2} \right). \]
 If $j$ is sufficiently large, by Lemma \ref{thm:jason}, we have $\mathbb{P}\left(A_j \right) \leq 5 n 2^{j} \exp(- c_d c^{\frac{2}{d}} j)$.
Thus, as $c> \left(\frac{\ln(2)}{c_d} \right)^{\frac{d}{2}}$, we have $ \sum_{j=0}^{+ \infty} \mathbb{P} \left(A_j \right) < \infty$. Borel-Cantelli Lemma entails the existence of $\Omega_{n,1}$, an event of probability $1$, such that, on $\Omega_{n,1}$, there exists $J_1 \in \N$ for which, for all $j \geq J_1$ and for all $0 \leq k \leq n 2^j$,  $k' \in \, \{2k,2k \pm 1,2k \pm 2\}$,
\begin{equation}\label{eqn:demouniformmodulus}
|X_{j+1,k'}-X_{j,k}| \leq  c j^\frac{d}{2} \|X_{j+1,k'}-X_{j,k}\|_{L^2(\Omega)}.
\end{equation}

Now, let us fix a compact interval $I \subseteq [0,n]$. There exists $J_2 \in \N$ such that, for all $j \geq J_2$, $2 2^{-j} \leq \diam(I)$. In particular, it means that, for all $t \in I$ and for all $j \geq J_2$, $k_j^-(t) \in I$ or $k_j(t)^+ \in I$. In the sequel, for all such $t$ and $j$, we choose $k_j(t) \in \{k_j^-(t),k_j^+(t) \}$ such that $k_{j}(t) \in I$. On the event $\Omega^*$ given by Proposition \ref{prop:holdereggene}, we can write, for all $j_0 \geq J_2$
\[ X_d^{H(\cdot)}(t) = X_{j_0,k_{j_0}(t)} + \sum_{j \geq j_0} (X_{j+1,k_{j+1}(t)} -X_{j,k_{j}(t)} ).\]

Therefore, on the event $\Omega_n=\Omega_{n,1} \cap \Omega^*$ of probability $1$, if $s,t \in I$ are such that $2^{-(j_0+1)} \leq |s-t| \leq 2^{-j_0}$ for some $j_0 \geq \max \{J_1,J_2\}$, we write
\begin{align}
| X_d^{H(\cdot)}(t)-X_d^{H(\cdot)}(s)| &\leq |X_{j_0,k_{j_0}(t)}-X_{j_0,k_{j_0}(s)} | \nonumber\\
& \quad + \sum_{j \geq j_0} |X_{j+1,k_{j+1}(t)} -X_{j,k_{j}(t)} | \nonumber\\
& \quad +\sum_{j \geq j_0} |X_{j+1,k_{j+1}(s)} -X_{j,k_{j}(s)} |. \label{astuce}
\end{align}
From Proposition \ref{prop:holdofgene} and inequality \eqref{eqn:demouniformmodulus}, as $H$ satisfies Condition \ref{condi} (a), there is a constant $c_1$, only depending on $d$, $I$, $K$ and $c$ such that, for all $j \geq j_0$,
\[ \max \{|X_{j+1,k_{j+1}(t)} -X_{j,k_{j}(t)} |,|X_{j+1,k_{j+1}(s)} -X_{j,k_{j}(s)}| \} \leq c_1 j^{\frac{d}{2}} 2^{-\underline{H}(I)(j+1)}.\]
Also, as $|k_{J}(t)-k_{J}(s)| \leq 2 2^{-j_0}+ 2^{-j_0} \leq 2^{-j_0+2}$, we have, still from Proposition \ref{prop:holdofgene} and inequality \eqref{eqn:demouniformmodulus}, 
\[  |X_{J,k_{J}(t)}-X_{J,k_{J}(s)} | \leq c_1 j_0^{\frac{d}{2}} 2^{-\underline{H}(I)(j_0+2)}.\]
In total, we get the existence of a constant $c_2$, only depending on $d$, $I$, $K$ and $c$, such that, for all $s,t \in I$ with $2^{-(j_0+1)} \leq |s-t| \leq 2^{-j_0}$
\[ |X_d^{H(\cdot)}(t)-X_d^{H(\cdot)}(s) | \leq c_2 j_0^{\frac{d}{2}} 2^{-\underline{H}(I)j_0}  \leq c_2 |\log |s-t||^\frac{d}{2} |s-t|^{\underline{H}(I)}. \]
The conclusion follows by taking $\Omega^*_1 = \bigcap_n \Omega_n$.
\end{proof}

\section{Pointwise Hölder exponent}\label{sec:pointwiseholder}

Now, we want to show that the pointwise regularity of the process is governed by the Hurst function. Under Condition \ref{condi} (b), it is easy to show, with the help of Proposition \ref{prop:holdereggene}, that, almost surely, for any $t_0 \in \R_+$,
\[ h_{X_d^{H(\cdot)}} (t_0 ) \geq H(t_0).\]
To show the reverse inequality, we will prove the following Theorem.

\begin{Thm}\label{thm:lowerbound}
Given $d \in \N^*$, a compact set $K$ of $(\frac{1}{2},1)$ and a Hurst function $H \, : \, \R_+ \to K$ satisfying Condition \ref{condi} (b), there exists $\underline{\Omega}$, an event of probability $1$, such that, on $\underline{\Omega}$, for all $t_0 \in \R_+$, 
\begin{equation}\label{eq:limsuppos}
\limsup_{r \to 0^+}\frac{ \Osc(X_d^{H(\cdot)},[t_0-r,t_0+r] \cap \R_+ )}{r^{H(t_0)} (\log r^{-1})^\frac{-d^2 H(t_0)}{2(1-H(t_0))}} >0
\end{equation}
\end{Thm}

On this purpose, we use a generalization of a combination of previous ideas from the papers \cite{MR4108464,MR3922488,dawloosveldt}. First, remark that, similarly to the proof of Theorem \ref{thm:modulusofcontinuity}, it suffices to show that, for all $n \in \N$, there is $\underline{\Omega}_n$, an event of probability $1$, such that, on $\underline{\Omega}_n$, for all $t_0 \in [n,n+1]$, \eqref{eq:limsuppos} holds. The conclusion comes by taking $\underline{\Omega} = \bigcap_{n \in \N} \underline{\Omega}_n$. For the sake simpleness in notation, we prove this result for $n=0$.

Let us fix some notations. For all $(j,k) \in \N^2$, $\lambda_{j,k}$ stands for the dyadic interval of scale $j$ explicitly given by $[k2^{-j},(k+1)2^{-j})$, $\Lambda_j$ is the set of all dyadic intervals of scale $j$ and $\Lambda := \bigcup_j \Lambda_j$. If $\lambda=\lambda_{j,k} \in \Lambda_j$, $3\lambda$ is the set $\{ \lambda_{j,k-1},\lambda_{j,k}, \lambda_{j,k+1} \}$.  Finally, for all $j \in \N$ and $t_0 \in \R_+$, $\lambda_j(t_0)$ is the unique interval in $\Lambda_j$ such that $t_0 \in \lambda_j(t_0)$.  If $x \in [0,1]$ and $\lambda \in \Lambda$, we allow ourself to write $x \in 3 \lambda$ to state that there exists $\lambda' \in 3 \lambda$ for which $x \in \lambda'$. Similarly, $\lambda'' \subseteq 3 \lambda$ means that there exists  $\lambda' \in 3 \lambda$ for which $\lambda'' \subseteq \lambda'$.

For all $(j,k) \in \N \times \{0,\ldots,2^{j}-1\}$, let us set
\begin{align}
\Delta_{j,k}& :=X_d^{H(\cdot)}\left( \frac{k+1}{2^j} \right) - X_d^{H(\cdot)} \left( \frac{k}{2^j} \right)  \nonumber \\
 &= X_d \left( \frac{k+1}{2^j},H\left( \frac{k+1}{2^j} \right) \right) - X_d \left( \frac{k}{2^j},H\left( \frac{k}{2^j} \right) \right) . \nonumber
\end{align}
If $\lambda=\lambda_{j,k}$, we also write $\Delta_{\lambda}$ for $\Delta_{j,k}$. It is clear that, for all $j \in \N$,
\begin{equation}\label{eqn:defleader}
\sup_{\lambda \subseteq 3\lambda_j(t)} |\Delta_{\lambda}| \leq \Osc(X_d^{H(\cdot)},[t-22^{-j},t+22^{-j}] ).
\end{equation}
Recalling \eqref{eqn:def:gene} and the notation \eqref{nota:kernel}, we write
\begin{align*}
&  X_d^{H(\cdot)}\left( \frac{k+1}{2^j},H \left( \frac{k+1}{2^j}\right) \right) - X_d \left( \frac{k}{2^j},H\left( \frac{k+1}{2^j} \right)  \right) \nonumber \\
& \quad =I_d  \left(\int_{\frac{k}{2^j}}^{\frac{k+1}{2^j}} f_{H\left( \frac{k+1}{2^j} \right)}(s,\bullet) \, ds \right)  \nonumber \\
& \quad =I_d  \left(\mathbbm{1}_{(-\infty, \frac{k+1}{2^j}]^d}  \left(\int_{\frac{k}{2^j}}^{\frac{k+1}{2^j}} f_{H\left( \frac{k+1}{2^j} \right)}(s,\bullet) \, ds \right) \right)  \nonumber 
\end{align*}
since, as long as $s \in [\frac{k}{2^j},\frac{k+1}{2^j}]$, $f_{H\left( \frac{y+k}{2^j} \right)}(s,\mathbf{x})$ vanishes whenever $\mathbf{x} \notin (-\infty, \frac{k+1}{2^j}]^d$. The brilliant idea from \cite{MR4108464} is then to split this last integral in two parts, where one is ``negligible'' compared to the other one which enjoys some independence property.

\begin{Def}\label{def:chaptilde}
Given an integer $M \geq 0$, for all $(j,k) \in \N \times \{0,\ldots,2^{j}-1\}$, we consider the enlarged dyadic interval
\[\lambda_{j,k}^{M} :=\left( \frac{k-M}{2^j}, \frac{k+1}{2^j} \right]^d\]
and the random variables
\begin{equation}\label{def:variabledominante}
\widetilde{\Delta_{j,k}}^M  := I_d \left( \mathbbm{1}_{\lambda_{j,k}^{M}} \left(\int_{\frac{k}{2^j}}^{\frac{k+1}{2^j}} f_{H\left( \frac{k+1}{2^j} \right)}(s,\bullet) \, ds \right)  \right)
\end{equation}
and 
\[
    \widecheck{\Delta_{j,k}}^M  := I_d \left(\mathbbm{1}_{(-\infty, \frac{k+1}{2^j}]^d \setminus \lambda_{j,k}^{M}} \left(\int_{\frac{k}{2^j}}^{\frac{k+1}{2^j}} f_{H\left( \frac{k+1}{2^j} \right)}(s,\bullet) \, ds \right)  \right)
\]
\end{Def}
For all $(j,k) \in \N \times \{0,\ldots,2^{j}-1\}$, if we also set
\[ \widehat{\Delta_{j,k}}:=X_d \left( \frac{k}{2^j},H\left( \frac{k+1}{2^j} \right) \right)-X_d \left( \frac{k}{2^j},H\left( \frac{k}{2^j} \right) \right),\]
given $M \geq 0$, of course,  we have
\[ \Delta_{j,k} =  \widetilde{\Delta_{j,k}}^M + \widecheck{\Delta_{j,k}}^M  + \widehat{\Delta_{j,k}}.\]
Moreover, from the definition of Wiener-Itô integrals, we know that $\widetilde{\Delta_{j,k}}^M $ is measurable with respect to the $\sigma$-algebra
\[ \sigma(\{ B(t_2)-B(t_1) \, : \, t_1,t_2 \in \lambda_{j,k}^{M} \}), \]
see \cite[Lemma 2.1]{MR4108464}. Thus, if $M_1,\ldots,M_n$ are fixed positive real numbers, the random variables $\widetilde{\Delta_{j_1,k_1}}^{M_1}, \dots, \widetilde{\Delta_{j_n,k_n}}^{M_n}$ are independent as soon as the condition 
\begin{align}
\label{cond:indep}
     \ \lambda_{j_\ell,k_\ell}^{M_\ell} \cap  \lambda_{j_{\ell'},k_{\ell'}}^{M_{\ell'}} = \emptyset \mbox{ for all } 1 \leq \ell,\ell' \leq n
    \end{align}
is satisfied. 

Let us now give some lower and upper bounds for the norm in $L^2(\Omega)$ of these random variables. The following proposition is inspired by \cite[Lemmata 2.2 and 2.3;]{MR4108464} where the main modifications come from the fact that we are working here with a Hurst function instead of a constant Hurst parameter. 

\begin{Prop}\label{prop:l2norm}
Given $d \in \N^*$, a compact set $K$ of $(\frac{1}{2},1)$ and a Hurst function $H \, : \, \R_+ \to K$, there exists a positive deterministic constant $c$, only depending on $d$ and $K$, such that, for all $(j,k) \in \N \times \{0,\ldots,2^{j}-1\}$ and $M >0$, one has
\begin{enumerate}
\item $ c^{-1} 2^{- H\left( \frac{k+1}{2^j} \right) j}\leq \| \widetilde{\Delta_{j,k}}^M \|_{L^2(\Omega)} \leq c  2^{- H\left( \frac{k+1}{2^j} \right) j}$;
\item $\| \widecheck{\Delta_{j,k}}^M \|_{L^2(\Omega)} \leq c M^\frac{H\left( \frac{k+1}{2^j} \right)-1}{d} 2^{- H\left( \frac{k+1}{2^j} \right) j} $;
\item $\| \widehat{\Delta_{j,k}}\|_{L^2(\Omega)} \leq c \Osc(H,\lambda_{j,k})$.
\end{enumerate}
\end{Prop}

\begin{proof}
Let us start by showing the first point. Using the isometry property for Wiener-Itô integrals, we get, for all $\N \times \Z $ and $M >0$, with the changes of variable $s \mapsto 2^{-j}(u+k)$ and $\mathbf{w}= 2^j \mathbf{x}-k \mathbf{1}$,
\begin{align*}
 &\| \widetilde{\Delta_{j,k}}^M \|_{L^2(\Omega)}^2 = d! \int_{\lambda_{j,k}^{M}} \left(\int_{\frac{k}{2^j}}^{\frac{k+1}{2^j}} f_{H\left( \frac{k+1}{2^j} \right)}(s,\mathbf{x}) \, ds \right)^2 \, d\mathbf{x} \\
& \quad = d! 2^{- 2 H\left( \frac{k+1}{2^j} \right)j} \int_{(-M,1]^d} \left( \int_{0}^{1} f_{H\left( \frac{k+1}{2^j} \right)}(u,\mathbf{w}) \, du \right)^2 \, d\mathbf{w}.
\end{align*}
Let us remark that if $(u,\mathbf{w}) \in [0,1] \times [0,1]^d$, $\prod_{\ell=1}^d(u-w_\ell)_+ \in [0,1]$ and thus
\[ f_{H\left( \frac{k+1}{2^j} \right)}(u,\mathbf{w}) \geq f_{\sup K}(u,\mathbf{w}). \]
Therefore, we conclude
\[ \| \widetilde{\Delta_{j,k}}^M \|_{L^2(\Omega)} \geq \sqrt{d!} 2^{- H\left( \frac{k+1}{2^j} \right)j} \left \| \int_{0}^{1} f_{\sup K}(u,\bullet) \, du\right\|_{L^2([0,1]^d)}. \]
For the reverse inequality, it suffices to remark that
\begin{align*}
\int_{ (- \infty,1]^d} &\left(\int_{0}^1 f_{H\left( \frac{k+1}{2^j} \right)}(u,\mathbf{w}) \, du\right)^2 \, d\mathbf{w} \\&   \leq \int_{ (- \infty,1]^d} \left(\int_0^1 (f_{\inf K}(u,\mathbf{w}) + (f_{\sup K}(u,\mathbf{w})) du\right)^2 \, d\mathbf{w}.
\end{align*}

For the second point, let us write, for all $(j,k) \in \N \times \{0,\ldots,2^{j}-1\}$, \linebreak $H_{j,k}:=\frac{H\left( \frac{k+1}{2^j} \right)-1}{d}-\frac{1}{2}$. Again from the isometry property \eqref{eqn:isometry} for Wiener-Itô integrals, we get
\begin{align*}
& \| \widecheck{\Delta_{j,k}}^M \|_{L^2(\Omega)}^2 \\
& \leq d!  \int_{- \infty}^{\frac{k-M}{2^j}} (\frac{k}{2^j}-x_1)^{2H_{j,k}} dx_1 \times  \int_{\R^{d-1}} \left( \int_{\frac{k}{2^j}}^{\frac{k+1}{2^j}}\prod_{\ell=2}^d (s-x_\ell)_+^{H_{j,k}} \, ds \right)^2\, dx_2 \ldots dx_d . 
\end{align*}
First, we have
\[ \int_{- \infty}^{\frac{k-M}{2^j}} (\frac{k}{2^j}-x_1)^{2H_{j,k}} dx_1 = (M 2^{-j})^{2\frac{(H\left(\frac{k+1}{2^j} \right)-1)}{d} }\]
On the other hand, from the isometry property \eqref{eqn:isometry} for Wiener-Itô integrals and Proposition \ref{prop:holdofgene}, there exists a deterministic constant $c_1>0$, only depending on $d$, $K$ and $[0,1]$, for which
\begin{align*}
(&d-1)! \int_{\R^{d-1}} \left( \int_{\frac{k}{2^j}}^{\frac{k+1}{2^j}}\prod_{\ell=2}^d (s-x_\ell)_+^{H_{j,k}} \, ds \right)^2\, dx_2 \ldots dx_d \\
& \, = \left \|  X_{d-1}( \frac{k+1}{2^j}, \frac{(d-1)H\left( \frac{k+1}{2^j} \right)+1}{d})-X_{d-1}( \frac{k}{2^j}, \frac{(d-1)H\left( \frac{k+1}{2^j} \right)+1}{d}) \right\|_{L^2(\Omega)}^2
 \\
& \leq c_1 2^{-2j\frac{(d-1)H\left( \frac{k+1}{2^j} \right)+1}{d}}.
\end{align*}
In total, we have found a positive deterministic constant $c_2$, only depending on $d$ and $K$ and $[0,1]$, such that
\[ \| \widecheck{\Delta_{j,k}}^M \|_{L^2(\Omega)} \leq c_2 M^\frac{\sup K-1}{d} 2^{-H\left( \frac{k+1}{2^j} \right) j} . \]

The third and last point is a straightforward consequence of Proposition \ref{prop:holdofgene}.
\end{proof}

In view of the last property, we say that the random variables of the form $\widetilde{\Delta}_\lambda^M$ are dominant.

Finally, let us recall the following important fact about random variables in a given Wiener chaos, see \cite[Theorem 6.9 and Remark 6.10]{MR1474726} for a proof.

\begin{Lemma}\label{lemma:janson}
Given $d \in \N^*$, there exists an universal deterministic constant $\gamma_d \in [0,1)$ such that, for any random variable $X$ in the Wiener chaos of order $d$, one has
\[ \mathbb{P}\left(|X| \leq \frac{1}{2} \|X\|_{L^2(\Omega)}\right) \leq \gamma_d.\]
\end{Lemma}

We now have enough material to give a lower bound for the oscillations of the multifractional Hermite process.

\begin{proof}[Proof of Theorem \ref{thm:lowerbound}]
As already explained, we can reduce our attention to the interval $[0,1)$.

If $c_d>0$ is the constant in Lemma \ref{thm:jason}, we fix $c'> \left(\frac{\ln(2)}{c_d}\right)^\frac{d}{2}$. Let also $c$ be the constant given by Proposition \ref{prop:l2norm}. For all $\lambda_{j,k} \in \Lambda$,  we define
\begin{equation}\label{eqn:expressionforM}
M_\lambda := (8c^2c' j^\frac{d}{2})^{ \frac{d}{1-H\left( \frac{k+1}{2^j} \right)}}.
\end{equation}

First, we consider the dominant random variables. We need to fix some notations. If $\lambda=\lambda_{j,k}$ is a dyadic interval and $m \in \N$, $\mathcal{S}_{\lambda,m}=\mathcal{S}_{j,k,m}$ stands for the finite set of cardinality $2^m$ whose elements are the dyadic intervals of scale $j + m$ included in $\lambda_{j,k}$, formally speaking $\mathcal{S}_{j,k,m} := \{ \lambda \in \Lambda_{j+m} \, : \, \lambda \subset \lambda_{j,k} \}$.

If $\gamma_d \in [0,1)$ is the constant given in Lemma \ref{lemma:janson}, one can find $\ell_d \in \N$ such that 
\begin{equation}\label{eqn:gamma}
\gamma^{\ell_d} < 2^{-1}.
\end{equation}

If the dyadic interval $\lambda_{j,k}$ and $m \in \N$ are fixed and $S \in \mathcal{S}_{j,k,m}$ we define the sequences of dyadic intervals $(I_n)_{0 \leq n \leq m}$ and $(T_n)_{1 \leq n \leq m}$ in the following way: 
\begin{itemize}
\item $I_0=\lambda_{j,k}$:
\item  $I_m=S$;
\item for all $1 \leq n \leq m$, $I_{n-1}=I_n \cup T_n$.
\end{itemize}
Now, for any $1 \leq n \leq m$, there are $\ell_d$ dyadic intervals $(T_n^\ell=\lambda_{j_{n}^{(\ell)},k_{n}^{(\ell)}})_{1 \leq \ell \leq \ell_d}$ in $S_{T_n,\lfloor \log_2(\ell_d M_{T_n}) \rfloor +1}$ such that, for all $ 1 \leq \ell \leq \ell_d$
\[ \left(\frac{k_n^{(\ell)}-M_{T_n}}{2^{j_n^{(\ell)}}},\frac{k_n^{(\ell)}+1}{2^{j_n^{(\ell)}}} \right) \subseteq T_n \]
and, if $\ell \neq \ell'$, 
\[ \left(\frac{k_n^{(\ell)}-M_{T_n}}{2^{j_n^{(\ell)}}},\frac{k_n^{(\ell)}+1}{2^{j_n^{(\ell)}}} \right) \cap \left(\frac{k_n^{(\ell')}-M_{T_n}}{2^{j_n^{(\ell')}}},\frac{k_n^{(\ell')}+1}{2^{j_n^{(\ell')}}} \right) = \emptyset.  \]
Therefore, the dyadic intervals $(T_n^\ell)_{1 \leq n \leq m}^{1 \leq \ell \leq \ell_d}$ satisfy condition \eqref{cond:indep} (with $M_n=M_{T_n}$) which insures the independence of the random variables $(\widetilde{\Delta_{T_n^\ell}}^{M_{T_n}})_{1 \leq n \leq m}^{1 \leq \ell \leq \ell_d}$.

From this, for all $S \in \mathcal{S}_{j,k,m}$ we define the Bernoulli random variable
\[ \mathcal{B}_{j,k,m}(S) = \prod_{1 \leq n \leq m, 1 \leq \ell \leq \ell_d} \mathbbm{1}_{\{|\widetilde{\Delta_{T_n^\ell}}^{M_{T_n}}| < 2^{-1} \|\widetilde{\Delta_{T_n^\ell}}^{M_{T_n}}\|_{L^2(\Omega)}\}}.\]
Using Lemma \ref{lemma:janson} and the independence of the random variables $(\widetilde{\Delta_{T_n^\ell}}^{M_{T_n}})_{1 \leq n \leq m}^{1 \leq \ell \leq \ell_d}$, we conclude
\[ \E [\mathcal{B}_{j,k,m}(S)] \leq \gamma^{m \ell_d}. \]
Therefore, if we define the random variable
\[ \mathcal{G}_{j,k,m} = \sum_{S \in \mathcal{S}_{j,k,m}} \mathcal{B}_{j,k,m}(S)  \]
then $\E[\mathcal{G}_{j,k,m}] \leq (2 \gamma^{\ell_d})^m$. It follows from inequality \eqref{eqn:gamma} and Fatou Lemma that
\[ \E \left[ \liminf_{m \to + \infty} \mathcal{G}_{j,k,m} \right] = 0. \]
As a consequence,
$ \Omega_1 = \bigcap_{j \in \N, 0 \leq k < 2^{j}} \{ \omega \, : \, \liminf_{m \to + \infty} \mathcal{G}_{j,k,m}(\omega)=0 \} $
is an event of probability $1$.

Now if $\omega \in \Omega_1$ and $t_0 \in [0,1)$, we take $j \in \N$ and $k=k_j(t_0)$ and since, for all $m$, $\mathcal{G}_{j,k_j(t_0),m}$ has values in $\{0,\ldots,2^m \}$ we conclude that there are infinitely many $m$ for which, for every $S \in \mathcal{S}_{j,k_{j}(t_0),m}$, $\mathcal{B}_{j,k,m}(S)=0$. Considering such a $m$ and $S= \lambda_{j+m}(t_0)$ then we first remark that, for all $1 \leq n \leq m$, $I_n=\lambda_{j+n}(t_0)$ and thus $T_n \in 3\lambda_{j+n}(t_0)$. Now, as $\mathcal{B}_{j,k,m}(\lambda_{j+m}(t_0))=0$, one can find $1 \leq n \leq m$ and $1 \leq \ell \leq \ell_d$ such that
\[|\widetilde{\Delta_{T_n^\ell}}^{M_{T_n}}(\omega)| \geq  \frac{1}{2} \|\widetilde{\Delta_{T_n^\ell}}^{M_{T_n}}\|_{L^2(\Omega)}. \]

In short, we have shown that for all $\omega \in \Omega_1$ and $t_0\in [0,1)$, there exist infinitely many $j \in \N$ such that there is $\lambda \in 3 \lambda_j(t_0)$ and $\lambda' \in  \mathcal{S}_{\lambda,\lfloor \log_2(\ell_d M_{\lambda}) \rfloor +1}$ for which
\begin{equation}\label{eq:osc1}
|\widetilde{\Delta_{\lambda'}^{M_{\lambda}}}(\omega)| \geq  \frac{1}{2} \|\widetilde{\Delta_{\lambda'}^{M_{\lambda}}}\|_{L^2(\Omega)}.
\end{equation}

On the other hand, from Lemma \ref{thm:jason}, we know that, for all $j$ large enough,
\begin{align*}
\mathbb{P}&\left(\exists \lambda \in \Lambda_{j}, \lambda' \in \mathcal{S}_{\lambda,\lfloor \log_2(\ell_d M_{\lambda}) \rfloor +1}  \, : \,  \left| \widecheck{\Delta_{\lambda'}}^{M_\lambda} \right| \geq  c' j^\frac{d}{2}  \left \| \widecheck{\Delta_{\lambda'}}^{M_{\lambda}} \right\|_{L^2(\Omega)}  \right)\\
& \leq  2 \ell_d \sup_{\lambda \in \Lambda_j} M_\lambda 2^j \exp(-c_d (c')^\frac{2}{d} j).
\end{align*}
Thus, as $c'> \left(\frac{\ln(2)}{c_d}\right)^\frac{d}{2}$, recalling the explicit expression \eqref{eqn:expressionforM}, we have
\[ \sum_{j=0}^{+ \infty}\mathbb{P}\left(\exists \lambda \in \Lambda_{j}, \lambda' \in \mathcal{S}_{\lambda,\lfloor \log_2(\ell_d M_{\lambda}) \rfloor +1}  \, : \,  \left| \widecheck{\Delta_{\lambda'}}^{M_\lambda} \right| \geq  c' j^\frac{d}{2}  \left \| \widecheck{\Delta_{\lambda'}}^{M_{\lambda}} \right\|_{L^2(\Omega)}  \right)  < \infty\]
We can then deduce from Borel-Cantelli Lemma the existence of $\Omega_2$, an event of probability $1$ such that, for all $\omega \in \Omega_2$, there exists $J_2 \in \N$ such that, for all $j \geq J_2$, $ \lambda \in \Lambda_{j}$ and $\lambda' \in S_{\lambda,\lfloor \log_2(\ell_d M_{\lambda}) \rfloor +1}$,
\begin{equation}\label{eq:osc2}
 \left| \widecheck{\Delta_{\lambda'}}^{M_\lambda} (\omega) \right| \leq c' j^\frac{d}{2} \left \| \widecheck{\Delta_{\lambda'}}^{M_\lambda} \right\|_{L^2(\Omega)}.
\end{equation}

Similarly, we prove the existence of $\Omega_3$, an event of probability $1$ such that, for all $\omega \in \Omega_3$, there exists $J_3 \in \N$ such that, for all $j \geq J_3$, $ \lambda \in \Lambda_{j}$ and $\lambda' \in S_{\lambda,\lfloor \log_2(\ell_d M_{\lambda}) \rfloor +1}$,
\begin{equation}\label{eq:osc3}
 \left| \widehat{\Delta_{\lambda'}} (\omega) \right| \leq c' j^\frac{d}{2} \left \| \widehat{\Delta_{\lambda'}} \right\|_{L^2(\Omega)}.
\end{equation}

Now, if $\omega$ is such that inequalities \eqref{eq:osc1}, \eqref{eq:osc2} and \eqref{eq:osc3} hold, with $ \lambda \in 3\lambda_{j}(t_0)$ and $\lambda'=\lambda_{j',k'} \in S_{\lambda,\lfloor \log_2(\ell_d M_{\lambda}) \rfloor +1}$ then, from Proposition \ref{prop:l2norm}, we deduce
\begin{align*}
|&\Delta_{\lambda'}(\omega)|  \geq |\widetilde{\Delta_{\lambda'}^{M_{\lambda}}}(\omega)| -\left( \left| \widecheck{\Delta_{\lambda'}}^{M_\lambda} (\omega) \right|  + \left| \widehat{\Delta_{\lambda'}} (\omega) \right| \right)  \\
& \geq \frac{c^{-1}}{2} 2^{-j' H\left( \frac{k'+1}{j'} \right)} - c c' \left(  (8c^2c' j^\frac{d}{2})^{\frac{H\left( \frac{k'+1}{2^{j'}}\right)-1}{1-H\left( \frac{k+1}{2^j} \right)}} j^\frac{d}{2} 2^{-j' H\left( \frac{k'+1}{j'} \right)}+\Osc(H,\lambda')\right).
\end{align*}
First, by Condition \ref{condi} (b), we know that there is $c''>0$  and $\gamma > H(t_0)$ such that
\begin{equation}\label{rg:H:regponc}
\Osc(H,\lambda') \leq c'' 2^{-\gamma j}.
\end{equation}
Let us remark that
\begin{align*}
\lim_{j \to + \infty} j^{\frac{d}{2} \left(1+\frac{H\left( \frac{k'+1}{2^{j'}}\right)-1}{1-H\left( \frac{k+1}{2^j} \right)} \right)} = \lim_{j \to + \infty} \exp \left(\ln(j) \frac{d}{2} \left( \frac{H\left( \frac{k'+1}{2^{j'}}\right)-H\left( \frac{k+1}{2^j} \right)}{1-H\left( \frac{k+1}{2^j} \right)} \right) \right)
\end{align*}
and, as
\begin{align*}
\left| \frac{H\left( \frac{k'+1}{2^{j'}}\right)-H\left( \frac{k+1}{2^j} \right)}{1-H\left( \frac{k+1}{2^j} \right)} \right| \leq c''\frac{ 2^{-\gamma j}}{1-\sup K},
\end{align*}
we get
\[ \lim_{j \to + \infty} j^{\frac{d}{2} \left(1+\frac{H\left( \frac{k'+1}{2^{j'}}\right)-1}{1-H\left( \frac{k+1}{2^j} \right)} \right)} =1\]
Similarly, we also have
\[ \lim_{j \to + \infty} (8c^2c')^{\frac{H\left( \frac{k'+1}{2^{j'}}\right)-1}{1-H\left( \frac{k+1}{2^j} \right)}}= (8c^2c')^{-1} \]
In particular, if $j$ is large enough, $H\left( \frac{k'+1}{j'} \right)$ is also strictly less that $\gamma$ and one can write
\[ c c'\Osc(H,\lambda') \leq \frac{c^{-1}}{8} (8c^2c' j^\frac{d}{2})^{ -\frac{dH\left( \frac{k'+1}{j'} \right)}{1-H\left( \frac{k+1}{2^j} \right)}} 2^{-j H\left( \frac{k'+1}{j'} \right)}.\]
Putting all of these together, we conclude that, for all $j$ sufficiently large,
\begin{align}
|\Delta_{\lambda'}(\omega)|  &\geq \frac{c^{-1}}{4} 2^{-j' H\left( \frac{k'+1}{j'} \right)}- \frac{c^{-1}}{8} (8c^2c' j^\frac{d}{2})^{ -\frac{dH\left( \frac{k'+1}{j'} \right)}{1-H\left( \frac{k+1}{2^j} \right)}} 2^{-j H\left( \frac{k'+1}{j'} \right)} \nonumber \\
& \geq \frac{c^{-1}}{8} (8c^2c' j^\frac{d}{2})^{ -\frac{dH\left( \frac{k'+1}{j'} \right)}{1-H\left( \frac{k+1}{2^j} \right)}} 2^{-j H\left( \frac{k'+1}{j'} \right)}.\label{eqn:finpreuve}
\end{align}

In total, on $\Omega_1 \cap \Omega_2 \cap \Omega_3$, which is an event of probability $1$, for all $t_0 \in [0,1)$, we have, from equations \eqref{eqn:defleader}, \eqref{eqn:finpreuve} and Condition \ref{condi} (b) for $H$,
\[ \limsup_{j \to + \infty} \frac{\Osc(X_d^{H(\cdot)},[t_0-22^{-j},t_0+22^{-j}] \cap \R_+ )}{2^{-j H(t_0)} j^{- \frac{d^2H(t_0)}{2(1-H(t_0))}}}>0.\]
\end{proof}

Theorem \ref{thm:holderregularity} is then a direct consequence of Proposition \ref{prop:holdereggene} and Theorem \ref{thm:lowerbound}.

\begin{proof}[Proof of Theorem \ref{thm:holderregularity}]
Let us consider the events $\Omega^*$ and $\Omega^*_2$ given by Proposition \ref{prop:holdereggene} and Theorem \ref{thm:lowerbound} respectively. Then, $\Omega^* \cap \underline{\Omega}$ is an event of probability $1$ on which:
\begin{itemize}
\item for all $t \geq 0$, $h_{X_d^{H(\cdot)}}(t_0) \geq H(t_0)$, by Proposition \ref{prop:holdereggene} and Condition \ref{condi} (b);
\item for all $t \geq 0$, $h_{X_d^{H(\cdot)}}(t_0) \leq H(t_0)$, by Theorem \ref{thm:lowerbound}.
\end{itemize}
\end{proof}

\begin{Rmk}
Let us recall that, for all continuous function $f$, all interval $I$ and all $t_0 \in I$, $H_f(I) \leq h_f(t_0)$. Thus, an immediate consequence of Theorems \ref{thm:modulusofcontinuity} and \ref{thm:holderregularity} is that, if the Hurst function $H$ satisfies Conditions \ref{condi} (a) and (b), on the event $\Omega^*_1 \cap \Omega^*_2$ of probability $1$, for all interval $I \subseteq \R_+$
\begin{equation}\label{eqn:exposantholderunif}
H_{X_d^{H(\cdot)}} (I) = \underline{H}(I).
\end{equation}
Let us note that only Condition (a) is required to deduce this fact. Indeed, if $t_0 \in I$ is such that $H(t_0) = \underline{H}(I)$, then it is still possible possible to reach the bounds \eqref{rg:H:regponc} at $t_0$ and then \eqref{eqn:finpreuve}. Then, on $\underline{\Omega}$, \eqref{eq:limsuppos} holds at $t_0$. It follows that, on $\underline{\Omega}$
\[  H_{X_d^{H(\cdot)}} (I) \leq h_{X_d^{H(\cdot)}} (t_0) \leq H(t_0)= \underline{H}(I)\]
and the equality \eqref{eqn:exposantholderunif} holds on $\underline{\Omega} \cap \Omega^*$.
\end{Rmk}

\section{Law of iterated logarithm}\label{section:iterated}

Let us now prove that multifractional Hermite processes enjoy a law of iterated logarithm. We use similar arguments as in Sections \ref{sect:modulus} and \ref{sec:pointwiseholder} but somehow ``localize'' them. This localization helps us to deduce, at each point, a sharper modulus of continuity which bounds, both from above and below, the oscillations of the process around this point. Let us start by showing the positiveness of the limit in \eqref{eq:iterated}.

\begin{Prop}\label{pro:iterated1}
Given $d \in \N^*$, a compact set $K$ of $(\frac{1}{2},1)$ and a Hurst function $H \, : \, \R_+ \to K$ satisfying Condition \ref{condi} (c), there exists $\overline{\Omega}_1$, an event of probability $1$, such that on $\overline{\Omega}_1$, for (Lebesgue) almost every $t_0 \in \R_+$, we have
\begin{equation}\label{eqn:prop1iterad}
 \limsup_{r \to 0^+} \frac{ \Osc(X_d^{H(\cdot)},[t_0-r,t_0+r] \cap \R_+ )}{r^{H(t_0)} (\log(\log r^{-1}))^\frac{d}{2}} < \infty.
\end{equation}
\end{Prop}

\begin{proof}
We use the notation introduced before the proof of Theorem \ref{thm:modulusofcontinuity}. Let us fix $t_0 \in [0,1)$ and $c > c_d^\frac{-2}{d}$, with $c_d>0$ the constant in Lemma \ref{thm:jason}. For all $j_0\in \N$, let $A_{j_0}(t)$ be the event defined by
\[ \left(\exists j \geq j_0, \lambda_{k,j}, \lambda_{k',j} \subseteq 3\lambda_{j_0}(t_0) : \frac{|X_{j,k'}-X_{j,k}|}{\|X_{j,k'}-X_{j,k}\|_{L^2(\Omega)}} \geq c \log(j_0)^{\frac{d}{2}} (j-j_0+1)^\frac{d}{2}  \right). \]
 If $j_0$ is sufficiently large, we have, by Lemma \ref{thm:jason}, 
\begin{align*}
\mathbb{P}(A_{j_0}(t)) &\leq \sum_{j \geq j_0} 3 2^{j-j_0}\exp(-c_d c^\frac{2}{d} \log(j_0) (j-j_0+1)) \\
& \leq c' \exp(-c_d c^\frac{2}{d} \log(j_0)),
\end{align*} 
for a deterministic constant $c'>0$ independent of any relevant quantities. Thus, as $c > c_d^\frac{-2}{d}$, we have $
\sum_{j_0=0}^{+ \infty}   \mathbb{P}(A_{j_0}(t))  < \infty$ and Borel-Cantelli Lemma entails the existence of $\Omega_{t_0}$, an event of probability $1$, such that, on $\Omega_{t_0}$, there exists $J \in \N$, for which, for all $j \geq j_0 \geq J$, $ \lambda_{k,j}, \lambda_{k',j}  \subseteq  3\lambda_{j_0}(t_0)$,
\begin{equation} \label{eq:demoupperbouniter}
|X_{j,k'}-X_{j,k}| \leq c \log(j_0)^{\frac{d}{2}} (j-j_0+1)^\frac{d}{2} \|X_{j,k'}-X_{j,k}\|_{L^2(\Omega)}.
\end{equation}

Let us then consider $j_0 \geq J$ and $s, t \in [t_0-r,t_0+r]$ with $2^{-(j_0+1)} \leq r \leq 2^{-j_0}$. For any $j \geq j_0$ and $x \in \{s,t\}$, $\lambda_j(x) \subseteq 3 \lambda_{j_0}(t_0)$. Thus, increasing $j_0$ if necessary, from Proposition \ref{prop:holdofgene} and the Condition \ref{condi} (c) for $H$, one can write,
\begin{align}
\|X_{j,k_{j}(t)}-X_{j,k_{j}(s)}\|_{L^2(\Omega)} & \leq c_1 \left(2^{-j \min\{H(k_{j}(t)2^{-j}),H(k_{j}(s)2^{-{j}})\}} + 2^{-j H(t_0)} \right) \nonumber \\ 
& \leq c_1 \left(2^{-j \min\{H(t),H(s)\}} 2^{-j 2^{-j H(t_0)}} + 2^{-jH(t_0)}\right) \nonumber \\
& \leq 2 c_1 2^{-j \min\{H(t_0),H(t),H(s)\}} \label{eq:oufouf}
\end{align}
for a deterministic constant $c_1>0$, only depending on $d$, $K$ and $[0,1]$.

On the event $\Omega_*$ given by Proposition \ref{prop:holdereggene}, one can write
\begin{align*}
X_d^{H(\cdot)}(t) - X_d^{H(\cdot)}(s)& = X_{j_0,k_{j_0}(t)}-X_{j,k_{j_0}(s)}  \\ & + \sum_{j \geq j_0} \left(X_{j+1,k_{j+1}(t)}-X_{j+1,k_{j+1}(s)}-X_{j,k_{j}(t)}+X_{j,k_{j}(s)} \right). 
\end{align*}
It then follows from inequalities \eqref{eq:demoupperbouniter} and \eqref{eq:oufouf} that, on $\Omega^* \cap \Omega_{t_0}$, there exists a constant $c_2>0$, only depending on $d$, $K$ and $[0,1]$ such that
\[ |X_d^{H(\cdot)}(t) - X_d^{H(\cdot)}(s)| \leq c_2 2^{-j_0 \min\{H(t_0),H(t),H(s)\}} \log(j_0)^{\frac{d}{2}}.\]
Increasing $j_0$ if necessary, the Condition \ref{condi} (c) for $H$ and the inequalities $2^{-(j_0+1)} \leq r \leq 2^{-j_0}$ finally give
\[ |X_d^{H(\cdot)}(t)-X_d^{H(\cdot)}(s) | \leq 2 c_2 r^{H(t_0) } (\log(\log(r^{-1}))^\frac{d}{2}.\]

In total, we have proved that for any $t_0 \in [0,1]$ there exists $\Omega_{t_0}$, an event of probability $1$, on which \eqref{eqn:prop1iterad} holds. The conclusion follows by countable intersection and Fubini theorem.
\end{proof}

Let us now focus on the positiveness of the limit in \eqref{eq:iterated}. We use again the random variables introduced in Definition \ref{def:chaptilde}. First, we need to bound from below the probabilities
\begin{equation}\label{eqn:probatoboundfrombelow}
\mathbb{P}(|\widetilde{\Delta}^M_{j,k}| \geq y 2^{-j H\left( \frac{k+1}{2^j}\right)})
\end{equation}
for $(j,k) \in \N \times \{0,\ldots,2^{j}-1\}$ and $M>0$. We know that for any variable $X$ in the Wiener chaos of order $d$, there exist two deterministic constants $y_0 \geq 0$ and $c>0$ such that, for all $y \geq y_0$,
\[ \mathbb{P}(|X| \geq y) \geq \exp(-c y^\frac{2}{d}), \]
see \cite[Theorem 6.12]{MR1474726}. But, unfortunately, these constants depend on the law of $X$ and are not universal, which is undesirable in our context. Nevertheless, using some convergences in $L^2(\Omega)$, we still manage to ``uniformly'' bound the probabilities \eqref{eqn:probatoboundfrombelow} from below.

\begin{Lemma}\label{lemma:lowerboundinproba}
Let $d \in \N^*$, $K$ be a compact set of $(\frac{1}{2},1)$ and $H \, : \, \R_+ \to K$ be a continuous Hurst function. For all $t_0 \in \R_+$, there exist four deterministic constants $c_{t_0}>0$, $y_{t_0}>0$, $j_0 \in \N$ and $M_0>0$ such that, for all $ \lambda_{j,k} \subseteq 3 \lambda_{j_0}(t_0)$, $M \geq M_0$ and $y >y_{t_0}$, we have
\begin{equation}\label{eq:lowerboundinproba}
 \mathbb{P}(|\widetilde{\Delta}^M_{j,k}| \geq y 2^{-j H\left( \frac{k+1}{2^{j}}\right)}) \geq \exp(-c_{t_0} y^{\frac{2}{d}}).
\end{equation}
\end{Lemma}

\begin{proof}
For all $j \in \N$, $k \in \{0,\ldots 2^{j}-1\}$ and $M \in \N$, by auto-similarity and stationarity of increments for standard Hermite processes, we know that the random variables
\[ \widetilde{\Delta}^M_{j,k} + \widecheck{\Delta}^M_{j,k} \text{ and } 2^{-jH\left( \frac{k+1}{2^{j}}\right)} X_d\left(1,H \left(\frac{k+1}{2^{j}} \right)\right) \]
are equals in law.
We also know that there exist two deterministic constants $c_{t_0}^1>0$ and $y_{t_0}^1>0$ such that, for all $y>y_{t_0}^1$
\[ \mathbb{P}\left(|X_d(1,t_0)| \geq y \right) \geq \exp(-c_{t_0}^1 y^{\frac{2}{d}}).\]

For all $y>0$, we write
\begin{align*}
\mathbb{P}&\left((|\widetilde{\Delta}^M_{j,k}| \geq y 2^{-j H\left( \frac{k+1}{2^{j}}\right)}\right)\\ & \geq \mathbb{P}\left(\left|X_d\left(1,H \left(\frac{k+1}{2^{j}}\right)\right)\right| \geq 2y \right)- \mathbb{P}(|\widecheck{\Delta}^M_{j,k}| > y 2^{-j H\left( \frac{k+1}{2^{j}}\right)}).
\end{align*}
By Lemma \ref{thm:jason} and Proposition \ref{prop:l2norm}, we know that, for all $j,k$, and all $M$ and $y$ sufficiently large,
\begin{align*}
 \mathbb{P}\left(|\widecheck{\Delta}^M_{j,k}| > y 2^{-j H\left( \frac{k+1}{2^{j}}\right)}\right) &\leq \exp\left(-c_d \left(y c^{-1} M^{\frac{1-H\left(\frac{k+1}{2^{j}}\right)}{d}} \right)^\frac{2}{d} \right) \\
 & \leq \exp\left(-c_d \left(y c^{-1} M^{\frac{1- \sup K}{d}} \right)^\frac{2}{d} \right)
\end{align*}
As $1-\sup K>0$, if $M$ is large enough, one can then reach
\[ \mathbb{P}\left(|\widecheck{\Delta}^M_{j,k}| > y 2^{-j H\left( \frac{k+1}{2^{j}}\right)}\right) \leq \frac{1}{4} \exp(- 3 c_{t_0}^1 y^\frac{2}{d}).\]
On the other hand, we have
\begin{align*}
\mathbb{P} &\left(\left|X_d\left(1,H \left(\frac{k+1}{2^{j}} \right)\right)\right| \geq 2y \right) \\& \geq \mathbb{P}\left(\left|X_d(1,t_0)\right| \geq 3y \right)-\mathbb{P}\left(\left|X_d(1,H(t_0))-X_d\left(1,H \left(\frac{k+1}{2^{j}}\right)\right)\right| > y \right).
\end{align*}
Using again Lemma \ref{thm:jason}, from Proposition \ref{prop:holdofgene}, we know that there exists a deterministic constant $c_2>0$, only depending on $d$, $K$ and $[0,1]$, such that, for all $j_0$ large enough,
\begin{align*}
\mathbb{P}&\left(\left|X_d(1,H(t_0))-X_d\left(1,H \left(\frac{k+1}{2^{j}}\right)\right)\right| > y \right) \\ & \leq \exp \left( - c_d\left\|X_d(1,H(t_0))-X_d\left(1,H \left(\frac{k+1}{2^j}\right)\right)\right\|_{L^2(\Omega)}^\frac{-2}{d} y^\frac{2}{d} \right) \\
& \leq \exp \left( - c_d c_2^\frac{-2}{d} \left|H(t_0)-H\left( \frac{k+1}{2^{j}}\right) \right|^\frac{-2}{d} y^\frac{2}{d} \right).
\end{align*}
The continuity of $H$ insures that, if $j$ is large enough,
\[ \mathbb{P}\left(\left|X_d(1,H(t_0))-X_d\left(1,H \left(\frac{k+1}{2^j}\right)\right)\right| > y \right) \leq \frac{1}{4} \exp(- 3 c_{t_0}^1 y^\frac{2}{d}).\]
Putting everything together, we conclude the existence of $c_{t_0}>0$ and $y_{t_0}>0$ with the desired property.
\end{proof}

Let us use this last Lemma to prove the positiveness of the limit in \eqref{eq:iterated}.

\begin{Prop}\label{pro:iterated2}
Given $d \in \N^*$, a compact set $K$ of $(\frac{1}{2},1)$ and a Hurst function $H \, : \, \R_+ \to K$ satisfying Condition \ref{condi} (b), there exists $\overline{\Omega}_2$, an event of probability $1$, such that on $\overline{\Omega}_2$, for (Lebesgue) almost every $t_0 \in \R_+$, we have
\[
 0<\limsup_{r \to 0^+} \frac{ \Osc(X_d^{H(\cdot)},[t_0-r,t_0+r] \cap \R_+ )}{r^{H(t_0)} (\log(\log r^{-1}))^\frac{d}{2}}.
\]
\end{Prop}

\begin{proof}
We use the same notation as in the proof of Theorem \ref{thm:lowerbound}. Let us fix $t_0 \in [0,1)$. If $M$ and $j$ are sufficiently large such that \eqref{eq:lowerboundinproba} holds for all large enough $y$, then, for any $m \in \N$ and $\lambda \in \mathcal{S}_{\lambda_{j}(t_0),m}$, let  $(I_n)_{0 \leq n \leq m}$ and $(T_n)_{1 \leq n \leq m}$ be the sequences of dyadic intervals with $I_0=\lambda_j(t_0)$; $I_m=\lambda$ and, for all $1 \leq n \leq m$, $I_{n-1}=I_{n} \cup T_n$. For all $1 \leq n \leq m$, let also $T_n^\star=\lambda_{k_n,j_n} \in S_{T_n, \lfloor \log_2(M)\rfloor +1}$ such that $(T^{\star}_{n})^M \subseteq T_n$. In particular, for all $n \neq n'$, $ (T^{\star}_{n})^M \cap (T^{\star}_{n'})^M= \emptyset$ and the random variables $(\widetilde{\Delta}^M_{T^{\star}_{n}})_{1 \leq n \leq m}$ are independent.

If $c_{t_0}>0$ is the constant given in Lemma \ref{lemma:lowerboundinproba} and $c_1$ is a deterministic constant such that $0<c_{t_0} c_1^\frac{2}{d}<1$, let us consider the event
\[    \mathcal{E}_{j,m}(t_0)= \left \{ \omega \in \Omega \, : \, \max_{1 \leq n \leq m} \left| \frac{\widetilde{\Delta}^M_{T^{\star}_{n}}}{2^{-j_n H \left( \frac{k_n+1}{2^{j_n}}\right)}} \right| \geq c_1 \log(2m)^\frac{d}{2} \right\}.\]
Using the independence of the random variables $(\widetilde{\Delta}^M_{T^{\star}_{n}})_{1 \leq n \leq m}$, Lemma \ref{lemma:lowerboundinproba} and the inequality $\log(1-x) \leq -x$ for all $x \in (0,1)$, we get, if $m$ is large enough,
\begin{align*}
\mathbb{P}(\mathcal{E}_{j,m}(t_0))&=1- \prod_{n=1}^m \mathbb{P} \left( \left| \frac{\widetilde{\Delta}^M_{T^{\star}_{n}}}{2^{-j_n H \left( \frac{k_n+1}{2^{j_n}}\right)}} \right| < c_1 \log(2m)^\frac{d}{2}\right) \\
&  \geq 1-(1-\exp(-c_{t_0} c_1^\frac{2}{d} \log(2m))^m \\
&\geq 1-\exp \left( \frac{m}{(2m)^{c_{t_0}c_1^\frac{2}{d}}} \right)\\
& = 1- \exp \left( \frac{m^{1-c_tc_1^\frac{2}{d}}}{2^{c_{t_0}c_1^\frac{2}{d}}} \right).
\end{align*}
Thus, as $0<c_{t_0} c_1^\frac{2}{d}<1$, we get
\[ \sum_{p \in \N} \mathbb{P}(\mathcal{E}_{2^p,2^p}(t_0)) = \infty\]
and Borel-Cantelli Lemma, combined with the independence of the events $(\mathcal{E}_{2^p,2^p}(t_0))_p$ entails
\[ \mathbb{P} \left( \limsup_{p \to + \infty} \mathcal{E}_{2^p,2^p}(t_0) \right)=1.\]
In other words, there exists $\Omega_{t_0}^1$, an event of probability $1$, such that, for all $\omega \in \Omega_{t_0}^1$, there are infinitely many $j \in \N$ such that, there exist $\lambda \in 3 \lambda_j (t_0)$ and \linebreak $\lambda'_{j',k'} \in S_{\lambda,\lfloor \log_2(M) \rfloor +1}$ for which
\begin{equation}\label{maj1:itera}
|\widetilde{\Delta}^M_{\lambda'}(\omega)| \geq c_1 \log(j)^\frac{d}{2} 2^{-j' H\left( \frac{k'+1}{2^{j'}}\right)}.
\end{equation} 

On the other hand, if $c_2$ is a deterministic constant such that $c_2>c_d^\frac{-2}{d}$, we have, by Lemma \ref{thm:jason} and Proposition \ref{prop:l2norm}, for all $j \in \N$
\begin{align*}
\mathbb{P} &\left( \exists \lambda \in 3 \lambda_j(t_0), \lambda' \in S_{\lambda,\lfloor \log_2(M) \rfloor +1} \, : \, \left| \widecheck{\Delta}^M_{\lambda'}\right| \geq c_2 \log(j)^\frac{d}{2} \left\| \widecheck{\Delta}^M_{\lambda'}\right\|_{L^2(\Omega)} \right) \\
& \leq 3 M \exp(-c_d c_2^\frac{d}{2} \log(j)). 
\end{align*}
The fact that $c_2>c_d^\frac{-2}{d}$ and Borel-Cantelli Lemma entails the existence of $\Omega_{t_0}^2$, an event of probability $1$, such that, for all $\omega \in \Omega_{t_0}^2$, there exists $J \in \N$ for which, for all $j \geq J$, $\lambda \in 3 \lambda_j (t_0)$ and  $\lambda'_{j',k'} \in S_{\lambda,\lfloor \log_2(M) \rfloor +1}$,
\begin{equation}\label{maj2:itera}
|\widecheck{\Delta}^M_{\lambda'}(\omega)| \leq c_2 \left\| \widecheck{\Delta}^M_{\lambda'}\right\|_{L^2(\Omega)} \leq c_2 c M^\frac{H \left( \frac{k'+1}{2^{j'}}\right)-1}{d}\log(j)^\frac{d}{2} 2^{-j' H\left( \frac{k'+1}{2^{j'}}\right)},
\end{equation} 
where $c>0$ is the deterministic constant given by Proposition \ref{prop:l2norm}. Similarly, there exists $\Omega_{t_0}^3$, an event of probability $1$, such that, for all $\omega \in \Omega_{t_0}^3$, there exists $J \in \N$ for which, for all $j \geq J$, $\lambda \in 3 \lambda_j (t_0)$ and  $\lambda' \in S_{\lambda,\lfloor \log_2(M) \rfloor +1}$,
\begin{equation}\label{maj3:itera}
|\widehat{\Delta}_{\lambda'}(\omega)| \leq c_2 c \log(j)^\frac{d}{2} \Osc(H,\lambda').
\end{equation} 
As $\sup K <1$, by increasing $M$ if necessary, we can reach, $c_2c M^\frac{\sup K -1}{d} < \frac{c_1}{4}$. Also, from the Condition \ref{condi} (b) on $H$, there exits $c_3>0$ and $\gamma >H(t_0)$ such that, for all $\lambda' \in S_{\lambda,\lfloor \log_2(M) \rfloor +1}$,
\begin{align*}
\Osc(H,\lambda') & \leq c_3 2^{-j \gamma} \leq c_3 M^{\gamma} 2^{-j' \gamma}.
\end{align*}
Increasing $j$ if necessary, we can reach
\[ c_2 c\Osc(H,\lambda') \leq \frac{c_1}{4} 2^{-j' H\left( \frac{k'+1}{2^{j'}}\right)}.\]

If inequalities \eqref{maj1:itera}, \eqref{maj2:itera} and \eqref{maj3:itera} hold, we thus have, for all $M$ and $j$ big enough
\begin{align*}
|\Delta_{\lambda'}(\omega)| & \geq \frac{c_1}{2} \log(j)^\frac{d}{2} 2^{-j' H\left( \frac{k'+1}{2^{j'}}\right)} \\
& \geq \frac{c_1}{2} \log(j)^\frac{d}{2} M^{-H\left( \frac{k'+1}{2^{j'}}\right)}2^{-j H\left( \frac{k'+1}{2^{j'}}\right)}.
\end{align*}
In total, from the Condition \ref{condi} (b) for $H$ and inequality \eqref{eqn:defleader}, we deduce that, for all $t_0 \in [0,1)$, on the event $\Omega_{t_0}^1 \cap \Omega_{t_0}^2 \cap \Omega_{t_0}^3$ of probability $1$, we have
\[ \limsup_{j \to + \infty} \frac{\Osc(X_d^{H(\cdot)},[t_0-22^{-j},t_0+22^{-j}] )}{2^{-j H(t_0)} \log(j)^{\frac{d}{2}}}>0.\]
The conclusion follows again by countable intersection and Fubini theorem.
\end{proof}

Theorem \ref{thm:iterated} is then an immediate consequence of Propositions \ref{pro:iterated1} and \ref{pro:iterated2}.

\begin{proof}[Proof of Theorem \ref{thm:iterated}]
If $\overline{\Omega}_1$ and $\overline{\Omega}_2$ are the events of probability $1$ given by Proposition \ref{pro:iterated1} and \ref{pro:iterated2} respectively, on $\overline{\Omega}_1 \cap \overline{\Omega}_2$, we have, for (Lebesgue) almost every $t_0 \in \R_+$,
\[  0<\limsup_{r \to 0^+} \frac{ \Osc(X_d^{H(\cdot)},[t_0-r,t_0+r] \cap \R_+ )}{r^{H(t_0)} (\log(\log r^{-1}))^\frac{d}{2}} < \infty.\]

\end{proof}
 
\section{Local asymptotic self-similarity} \label{sec:locaasym}

Let us start this section by showing that the multifractional Hermite process $\{X_d^{H(\cdot)}(t) \}_{t \in \R_+}$ is weakly locally asymptotically self-similar. Our main ingredient is the following lemma, which is sometimes refereed as Slutsky's theorem (see for instance \cite[Page 318]{MR2059709}).

\begin{Lemma}\label{lemma:slutsky}
Let $(X_j)_j$ and $(Y_j)_j$ be two sequences of random variables such that $(X_j)_j$ converge in distribution to a random variable $X$ and $(Y_j)_j$ converges in probability to a deterministic constant $c$, then the sequence $(X_j+Y_j)_j$ converges in distribution to $X + c$.
\end{Lemma}

\begin{Prop}
Let $d \in \N^*$, $K$ be a compact set of $(\frac{1}{2},1)$ and $H \, : \, \R_+ \to K$ be a Hurst function  satisfying Condition \ref{condi} (b). For all $t_0 \geq 0$, the multifractional Hermite process $\{X_d^{H(\cdot)}(t) \}_{t \in \R_+}$ is weakly locally asymptotically self-similar of order $H(t_0)$ at $t_0$ with tangent process $\{X_d(t,H(t_0))  \}_{t \geq 0}$, the Hermite process of Hurst parameter $H(t_0)$.
\end{Prop}

\begin{proof}
Let us fix $t_0 \geq 0$. For all $t \geq 0$ and $\varepsilon>0$, we write
\begin{align*}
\varepsilon^{-H(t_0)} &\left( X_d^{H(\cdot)}(t_0+\varepsilon t)-X_d^{H(\cdot)}(t_0) \right) \\ &= \varepsilon^{-H(t_0)} \left( X_d(t_0+\varepsilon t,H(t_0+\varepsilon t))-X_d(t_0+\varepsilon t, H(t_0)) \right) \\ & \quad + \varepsilon^{-H(t_0)} \left( X_d(t_0+\varepsilon t, H(t_0))-X_d(t_0, H(t_0)) \right).
\end{align*}
First, from the well-known self-similarity and stationary of increments for the standard Hermite process, we know that the process
\[ \{ \varepsilon^{-H(t_0)} \left( X_d(t_0+\varepsilon t, H(t_0))-X_d(t_0, H(t_0)) \right)  \}_{ t \geq 0}\] 
is equal in finite-dimensional distribution to
\[ \{ X_d(t,H(t_0))  \}_{ t \geq 0}.\]
On the other hand, from Proposition \eqref{prop:holdofgene}, we know that, for all $t \geq 0$ and $\varepsilon>0$,
\begin{align*}
\| \varepsilon^{-H(t_0)} & \left( X_d(t_0+\varepsilon t,H(t_0+\varepsilon t))-X_d(t_0+\varepsilon t, H(t_0)) \right)\|_{L^2(\Omega)} \\
& \leq c_2 \varepsilon^{-H(t_0)} \left|H(t_0+\varepsilon t)-H(t_0) \right|. 
\end{align*}
If $t \geq 0$ is fixed, Condition \ref{condi} (b) insures that one can give find $\gamma> H(t_0)$ and $c>0$ such that, for all $\varepsilon>0$ sufficiently small
\[ \left|H(t_0+\varepsilon t)-H(t_0) \right| \leq c \varepsilon^{-\gamma}. \]
In particular, for all fixed $t \geq 0$, the sequence of random variables 
\[ \left(\varepsilon^{-H(t_0)}  \left( X_d(t_0+\varepsilon t,H(t_0+\varepsilon t))-X_d(t_0+\varepsilon t, H(t_0)) \right) \right)_{\varepsilon >0}\]
converges to $0$ in $L^2(\Omega)$, and thus in probability, when $\varepsilon \to 0^+$.

The conclusion follows from Lemma \ref{lemma:slutsky}.
\end{proof}

Now, we want to show that the local asymptotic self-similarity also holds in the strong sense. As already explained in Section \ref{sec:strategy}, it suffices to show that \eqref{eq:prohorov} holds, with $X=X_d^{H(\cdot)}$. On this purpose, we recall the Garsia-Rodemich-Rumsey inequality in the following lemma, a proof can be read in \cite{MR267632}.

\begin{Lemma}
Let $\Psi$ be a non-negative even function on $\R$ and $\rho$ be a non-negative even function on $[-1,1]$. Assume also that both $\Psi$ and $\rho$ are non decreasing on $\R_+$. If $f$ is a continuous function for which
\[\iint_{[0,1]^2} \Psi \left( \frac{f(x)-f(y)}{\rho(x-y)}  \right) \, dxdy \leq B < \infty,\]
then, for all $s,t \in [0,1]$,
\[ |f(s)-f(t)| \leq 8 \int_0^{|s-t|} \Psi^{-1} \left( \frac{4 B}{u^2} \right) \, d\rho(u).\]
\end{Lemma}
Applying this last Lemma to the functions $\Psi \, : \, u \mapsto |u|^p$ and $\rho \, : \,  \mapsto |u|^{\alpha+ \frac{1}{p}}$, for $p \geq 1$ and $\alpha \geq \frac{1}{p}$ then, with an obvious change of variable, we conclude that, for all $a >0$, there exists a deterministic constant $c_{a,p,\alpha}$ such that, for any $f \in C([0,a],\R)$ and $t \in [0,a]$
\begin{equation}\label{eqn:forlass}
|f(t)-f(0)|^p \leq c_{a,p,\alpha} t^{\alpha p-1} \iint_{[0,a]^2} |f(r)-f(v)|^p |r-v|^{-\alpha p-1} \, dr dv.
\end{equation}

We use this fact to prove Theorem \ref{thm:localasym}.

\begin{proof}[Proof of Theorem \ref{thm:localasym}]
It remains us to prove the strong local asymptotic self-similarity in the case where the Hurst function $H$ satisfies Condition \ref{condi} (c). Let us fix $a>0$ and $t_0 \geq 0$. For all $\varepsilon, \eta, \delta>0$, we set
\[ \mathbb{P}(\varepsilon,\eta,\delta) := \mathbb{P} \left( \sup_{s,t \in [0,a], |t-s| \leq \eta} \left| \frac{X^{H(\cdot)}(t_0+\varepsilon t)-X^{H(\cdot)}(t_0+\varepsilon s)}{\varepsilon^{H(t_0)}} \right| \geq \delta \right). \]
We have to show that, for all $\delta>0$, $\lim_{\eta \to 0^+} \limsup_{\varepsilon \to 0^+} \mathbb{P}(\varepsilon,\eta,\delta)=0$.

Of course, the Markov inequality entails, for any $p \geq 1$,
\[ \mathbb{P}(\varepsilon,\eta,\delta) \leq \delta^{-p} \varepsilon^{-p H(t_0)} \mathbb{E} \left[ \sup_{s,t \in [0,a], |t-s| \leq \eta} \left| X^{H(\cdot)}(t_0+\varepsilon t)-X^{H(\cdot)}(t_0+\varepsilon s) \right|^p\right].\]
Then, we use inequality \eqref{eqn:forlass} to write, for $\alpha \geq \frac{1}{p}$,
\begin{align*}
\mathbb{E} &\left[ \sup_{s,t \in [0,a], |t-s| \leq \eta} \left| X^{H(\cdot)}(t_0+\varepsilon t)-X^{H(\cdot)}(t_0+\varepsilon s) \right|^p\right] \\
& \leq c_{a,p,\alpha} \eta^{\alpha p-1} \iint_{[0,a]^2} \mathbb{E}\left[ \left| X^{H(\cdot)}(t_0+\varepsilon t)-X^{H(\cdot)}(t_0+\varepsilon s) \right|^p\right] |t-s|^{-\alpha p -1} \, ds dt.
\end{align*}
Moreover, we know from Corollary \ref{cor:lpnorm} that, for all $s,t \in [0,a]$,
\begin{align*}
\mathbb{E}&\left[ \left| X^{H(\cdot)}(t_0+\varepsilon t)-X^{H(\cdot)}(t_0+\varepsilon s) \right|^p\right] \\
& \leq \left( (\varepsilon |t-s|)^{\min \{H(t_0+ \varepsilon t),H(t_0 + \varepsilon s)\}} + |H(t_0+ \varepsilon t)- H(t_0 + \varepsilon s)|\right)^p.
\end{align*}
By Condition \ref{condi} (c), there exits a deterministic  constant $c>0$ such that, for all $\varepsilon>0$ sufficiently small and $s,t \in [0,a]$, $s \neq t$,
\[|H(t_0+ \varepsilon t)- H(t_0 + \varepsilon s)| \leq c (\varepsilon |t-s|)^{H(t_0)} \leq c \varepsilon^{H(t_0)} |t-s|^{\inf K} \]
and
\begin{align*}
(\varepsilon |t-s|)^{\min \{H(t_0+ \varepsilon t),H(t_0 + \varepsilon s)\}}\leq \varepsilon^{H(t_0)}  |t-s|^{\inf K} \varepsilon^{-(a \varepsilon)^{H(t_0)}} \leq 2 \varepsilon^{H(t_0)}  |t-s|^{\inf K}
\end{align*}
In total,  we have obtained that, for all $\varepsilon>0$ sufficiently small,
\begin{align*}
\mathbb{P}(\varepsilon,\eta,\delta) \leq 2 c_{a,p,\alpha} \delta^{-p}  \eta^{\alpha p-1} \iint_{[0,a]^2} |t-s|^{p(\inf K - \alpha)-1} \, ds dt.
\end{align*}
One can then choose, for instance, $\alpha = \frac{\inf K}{2}$ and $p= \frac{2}{\alpha}$ in order that the last integral is finite, because $p(\inf K - \alpha)-1=1$, and, as $\alpha > \frac{1}{p}$,
\[\lim_{\eta \to 0^+} \limsup_{\varepsilon \to 0^+} \mathbb{P}(\varepsilon,\eta,\delta)=0. \]
\end{proof}

\section{Fractal dimensions of the graph} \label{sec:dim}

Given a compact interval $I \subset \R_+$, let us start by providing an upper bound for the box-counting dimensions of the set $ \mathcal{G}_d(I)$. From the results proved in Sections \ref{sec:strategy} and \ref{sect:modulus}, it is in fact an easy task, thanks to the following lemma, see \cite[Corollary 11.2]{MR3236784} for a proof.

\begin{Lemma}\label{lemma:falco}
Let $I \subset \R_+$ be a compact interval and $f \, : \, I \to \R$ be a continuous function for which there exist $c \geq 0$ and $1 \leq \alpha \leq 2$ such that, for all $s,t \in I$,
\[ |f(s)-f(t)| \leq c |t-s|^{2-\alpha},\]
then
\[ \overline{\dim}_{\mathcal{B}} \left(\{ (t, X_d^{H(\cdot)}(t)) \, : \, t \in I\} \right) \leq \alpha.\]
\end{Lemma}
One can then directly state the following proposition.

\begin{Prop}\label{prop:upperbounddim}
Given $d \in \N^*$, a compact set $K$ of $(\frac{1}{2},1)$, a Hurst function $H \, : \, \R_+ \to K$ satisfying Condition \ref{condi} (a) and a compact interval $I \subset \R_+$, there exists $\widetilde{\Omega}_1$, an event of probability $1$, such that, on $\widetilde{\Omega}_1$, we have
\[ \overline{\dim}_{\mathcal{B}}\left( \mathcal{G}_d(I)\right)  \leq 2-\underline{H}(I).\]
\end{Prop}

\begin{proof}
It is an immediate consequence of Theorem \ref{thm:modulusofcontinuity} and Lemma \ref{lemma:falco}.
\end{proof}

To obtain a lower bound on the Hausdorff dimension, we use the notion of potential.

\begin{Def}
Let $d \in \N^*$ and $s \geq 0$. The \textit{$s$-potential} at a point $x \in \R^d$ to the measure $\mu$ on $\R^d$ is the quantity
\[ \phi_s(x) = \int \frac{d\mu(y)}{|x-y|^s}.\]
The \textit{$s$-energy} of $\mu$ is then defined as
\[ I_s(\mu) = \int \phi_s(x) \, d\mu(x) = \iint \frac{d\mu(x) d\mu(y)}{|x-y|^s}.\]
\end{Def}

Often, potentials and energies are used to get a lower bound for the Hausdorff dimension, as stated in the following Lemma, see \cite[Theorem 4.13]{MR3236784} for a proof.

\begin{Lemma}\label{lemma:falco2}
Let $d \in \N^*$ and $A$ be a subset of $\R^d$. If there is a measure $\mu$ on $A$ with $I_s(\mu) < \infty$, then $\mathcal{H}^s(A)= \infty$ and $\ha(A) \geq s$.
\end{Lemma}

In order to apply Lemma \ref{lemma:falco2} in our context, we have to find upper bounds for negative moments of the multifractional Hermite process. On this purpose, we use the following result, known as Carbery-Wright inequality, see \cite[Theorem 8]{MR1839474}.

\begin{Lemma}\label{lemma:carbery}
There is an absolute deterministic constant $c>0$ such that, for any $n,d \geq 1$, $1 < p < \infty$ any polynomial $Q \, : \, \R^d \to \R$ of degree at most $n$, any Gaussian random vector $(X_1,\ldots,X_d)$ and any $x>0$,
\[ \mathbb{E}[|Q(X_1,\ldots,X_d)|^\frac{p}{n} ]^\frac{1}{p} \mathbb{P}(|Q(X_1,\ldots,X_d)| \leq x ) \leq c p x^\frac{1}{n}.\]
\end{Lemma}

In the sequel, if a compact interval $I \subset \R_+$, $t_0 \in I$ and $\varepsilon>0$ are fixed, we set
\[\overline{H}(t_0,\varepsilon):=\overline{H}([t_0-\varepsilon,t_0+\varepsilon] \cap I) .\]

The proof of the following Lemma uses ideas from \cite[Lemma 14]{MR3192502} and \cite[Lemma 4.3]{MR3003367}, with modifications again mainly due to the fact that we are working with a non-constant Hurst function. As this Lemma is the main reason for the disparity between the lower and upper bounds for the fractal dimensions in Theorem \ref{thm:dim}, we believe that it is useful to write it in full details so that the readers can directly understand where it comes from.

\begin{Prop}\label{lemma:boundproba}
Given $d \in \N^*$, a compact set $K$ of $(\frac{1}{2},1)$, a Hurst function $H \, : \, \R_+ \to K$ satisfying Condition \ref{condi} (a) and a compact interval $I \subset \R_+$. If $t_0 \in I$ is such that $H(t_0)=\underline{H}(I)$, there exist two deterministic constants $c>0,\xi>0$, both only depending on $d$, $H$ and $I$, such that, for all $0 < \varepsilon < \xi$, $x \geq 0$ and $t,u \in I \cap [t_0- \varepsilon,t_0+\varepsilon]$,
\[ \mathbb{P}(|X_d^{H(\cdot)} (t)-X_d^{H(\cdot)} (u)| \leq x) \leq c x^\frac{1}{d}  |t-u|^{-\frac{\overline{H}(t_0,\varepsilon)}{d}}. \]
\end{Prop}

\begin{proof}
Let us fix $t,u\in I$. The symmetric function
\[ f_{t,u}^{H(\cdot)} \, : \, \R^d \to \R \, : \, \mathbf{w} \mapsto \int_0^t f_{H(t)}(s,\mathbf{w}) \, ds - \int_0^u f_{H(u)}(s,\mathbf{w}) \, ds   \]
belongs to $L^2(\R^d)$. By definition,
\[ I_d \left( f_{t,u}^{H(\cdot)} \right)= \left(X_d^{H(\cdot)} (t)-X_d^{H(\cdot)} (u) \right).\] 

Given $\{e_j\}_{j \in \N}$ an orthonormal basis of $L^2(\R)$, the sequence of functions
\[ \left( f_{t,u}^{H(\cdot),J} := \sum_{j_1, \ldots, j_d=1}^J \langle f_{t,u}^{H(\cdot)}, e_{j_1} \odot \cdots \odot e_{j_d} \rangle e_{j_1} \odot \cdots \odot e_{j_d} \right)_J, \]
where $\odot$ stands for the symmetric tensor product, converges to $f_{t,u}^{H(\cdot)}$ in $L^2(\R^d)$, see \cite[Appendix B.3.]{MR2962301}. 

For all $(j_1, \ldots ,j_d) \in \N^d$, we know that
\[ I_d \left( e_{j_1} \odot \cdots \odot e_{j_d}\right) = \prod_{\ell=1}^p H_{n_\ell}  \left(  \int_{\R} e_{ \widetilde{j_\ell}} (x) \, dB(x) \right),\]
where $n_\ell$ is the number of occurrence  of $\widetilde{j_\ell}$ in $(j_1, \ldots ,j_d) $ and $H_{n_\ell}$ is the Hermite polynomial of degree $n_\ell$, see \cite[Page 14]{MR2200233}. In particular, since, for all $f \in L^2(\R)$ with $\|f\|_{L^2(\R)}=1$, $I_1 (f) \sim \mathcal{N}(0,1)$, Lemma \ref{lemma:carbery} with $p=2d$ and $n=d$ entails, for all $J$,
\[ \mathbb{P}(|I_d \left( f_{t,u}^{H(\cdot),J} \right)| \leq x ) \leq c 2d x^\frac{1}{d}  \|I_d \left( f_{t,u}^J \right) \|_{L^2(\Omega)}^{-\frac{1}{d}}.  \]

Now, from the isometry property for Wiener-Itô integrals, we know that 
\[ I_d \left( f_{t,u}^{H(\cdot),J} \right) \underset{J \to \infty}{\longrightarrow} \left(X_d^{H(\cdot)} (t)-X_d^{H(\cdot)} (u)\right) \]
in $L^2(\Omega)$. In particular, there exists a subsequence $(f_{t,u}^{H(\cdot),J_k} )_{k}$ for which the convergence holds almost surely. Then, by Fatou's Lemma,
\begin{align*}
\mathbb{P}\left(|X_d^{H(\cdot)} (t)-X_d^{H(\cdot)} (u)| \leq x \right) &\leq \liminf_{k \to + \infty} \mathbb{P}\left(|I_d \left( f_{t,u}^{H(\cdot),J_k} \right)| \leq x \right) \\
& \leq c 2d x^\frac{1}{d} \liminf_{k \to + \infty}  \left\|I_d \left( f_{t,u}^{J_k} \right) \right\|_{L^2(\Omega)}^{-\frac{1}{d}} \\
& =  c 2d x^\frac{1}{d}  \left\| X_d^{H(\cdot)} (t)-X_d^{H(\cdot)} (u)  \right\|_{L^2(\Omega)}^{-\frac{1}{d}}.
\end{align*}

We use Proposition \ref{prop:holdofgene} to affirm that there exist two deterministic constants $c_1,c_2 >0$, only depending on $d$, $K$ and $I$, such that
\[ \left\| X_d^{H(\cdot)} (t)-X_d^{H(\cdot)} (u)  \right\|_{L^2(\Omega)} \geq c_1 |t-u|^{\min \{H(t),H(u) \}}-c_2|H(t)-H(u)|.\]
Since $H$ satisfies Condition \ref{condi} (a), there exists $\gamma> \underline{H}(I) = H(t_0)$ such that $H \in C^\gamma(I)$. Then, if $\xi>0$ is sufficiently small, for all $0 < \varepsilon < \xi$, we also have $\overline{H}(t_0,\varepsilon)<\gamma$. By reducing again $\xi>0$ if necessary, we have, for all $0<\varepsilon <\xi$ and $t,u \in I \cap [t_0- \varepsilon,t_0+\varepsilon]$
\[c_2|H(t)-H(u)| \leq \frac{c_1}{2} |t-u|^{\overline{H}(t_0,\varepsilon)}. \]
In total, we have obtained, for all such $\varepsilon$ and $t,u$,
\[ \mathbb{P}\left(|X_d^{H(\cdot)} (t)-X_d^{H(\cdot)} (u)| \leq x \right) \leq c 2d (2^{-1} c_1)^{-\frac{1}{d}} x^\frac{1}{d} |t-u|^{-\frac{\overline{H}(t_0,\varepsilon)}{d}}. \]
\end{proof}

We can now prove Theorem \ref{thm:dim}.

\begin{proof}[Proof of Theorem \ref{thm:dim}]
Let $t_0 \in I$ be such that $H(t_0)=\underline{H}(I)$. Let $\xi>0$ be given by Proposition \ref{lemma:boundproba} and $j \in \N^*$ with $\frac{1}{j}< \xi$. For all $t,r \geq 0$ such that $t,t+r \in [t_0-j^{-1},t_0+j^{-1}] \cap I$ and $s > 0$, we have, by Proposition \ref{lemma:boundproba},
\begin{align}\label{estimationespedim}
\mathbb{E} & \left[ \left(|X_d^{H(\cdot)}(t+r)-X_d^{H(\cdot)}(t)|^2+r^2 \right)^{- \frac{s}{2}} \right] \nonumber\\
& = \int_0^{r^{-s}}\mathbb{P}\left( \left(|X_d^{H(\cdot)}(t+r)-X_d^{H(\cdot)}(t)|^2+r^2 \right)^{- \frac{s}{2}} \geq x \right) \, dx \nonumber \\
&= s \int_0^{+ \infty} y (y^2+r^2)^{-\frac{s}{2}-1} \mathbb{P}\left( |X_d^{H(\cdot)}(t+r)-X_d^{H(\cdot)}(t)| \leq y  \right) \, dy \nonumber\\
& \leq c s \int_0^{+ \infty} y (y^2+r^2)^{-\frac{s}{2}-1} y^{\frac{1}{d}} r^{-\frac{\overline{H}(t_0,j^{-1})}{d}} \, dy \nonumber\\
& \leq c' r^{-\frac{\overline{H}(t_0,j^{-1})}{d}} \left(  r^{-s-2}\int_0^r y^{1+\frac{1}{d}} \, dy + \int_r^{+ \infty} y^{-s-1+\frac{1}{d}} \, dy \right) \nonumber\\
& \leq c'' r^{\frac{1}{d}-s-\frac{\overline{H}(t_0,j^{-1})}{d}}, 
\end{align}
where $c>0$ is given by Proposition \ref{lemma:boundproba} and $c',c''>0$ are deterministic constants only depending on $s$, $d$ and $c$.

Thus, if we consider the random measure $\mu_{X,j}$ defined for all Borel sets $A \subseteq \R^2$ by
\[ \mu_{X,j}(A) := \mathcal{L} \{ t \in [t_0-j^{-1},t_0+j^{-1}] \cap I \, : \, (t,X_d^{H(\cdot)}(t)) \in A \},\]
with $\mathcal{L}$ the Lebesgue measure in $\R$, we get
\begin{align*}
\mathbb{E} & \left( \iint \frac{d\mu_{X,j}(x) d\mu_{X,j}(y)}{|x-y|^s}\right) \\ &= \iint_{ ([t_0-j^{-1},t_0+j^{-1}] \cap I )^2} \mathbb{E} \left[ \left(|X_d^{H(\cdot)}(t)-X_d^{H(\cdot)}(u)|^2+|t-u|^2 \right)^{- \frac{s}{2}} \right] \, dt \, du \\
& \leq c''  \iint_{ ([t_0-j^{-1},t_0+j^{-1}] \cap I )^2} |t-u|^{\frac{1}{d}-s-\frac{\overline{H}(t_0,\varepsilon)}{d}} \, dt \, du.
\end{align*}
If $s < 1+\frac{1-\overline{H}(t_0,j^{-1})}{d}$, then this last integral is finite. Therefore, for all $q_j \in \Q$ with $0<q_j \leq 1+\frac{1-\overline{H}(t_0,j^{-1})}{d}$, there exists $\widetilde{\Omega}_{j,q_j}$, an event of probability $1$, such that on, $\widetilde{\Omega}_{j,q_j}$,
\[  \iint \frac{d\mu_{X,j}(x) d\mu_{X,j}(y)}{|x-y|^{q_j}} < + \infty.\]
By Lemma \ref{lemma:falco2} and \eqref{eqn:incredimension}, it means that, on $\widetilde{\Omega}_{j,q_j}$
\[ q_j \leq \ha\left( \mathcal{G}_d([t_0-j^{-1},t_0+j^{-1}] \cap I) \right) \leq \ha\left( \mathcal{G}_d(I) \right).\]
As $H$ is a continuous function, it follows that on the event $\bigcap_{j} \bigcap_{q_j} \widetilde{\Omega}_{j,q_j}$ of probability $1$, we have
\[ 1+\frac{1-\underline{H}(I)}{d} \leq \ha\left( \mathcal{G}_d(I) \right). \]
It suffices to intersect this event with $\Omega_1$ from Proposition \ref{prop:upperbounddim} to get the conclusion.
\end{proof}

\begin{Rmk}\label{rmk:redcuhyp}
Let us note that the proof or Proposition \ref{lemma:boundproba} only requires that, if $t_0 \in I$ is such that $H(t_0)=\underline{H}(I)$, there exist $\xi>0$ and $\gamma>0$ such that $\gamma > \overline{H}(t_0,\xi)$ and $H \in C^{\gamma}([t_0-\xi,t_0+\xi) \cap I)$. In particular the lower bound for the Hausdorff dimension of $\mathcal{G}_d(I)$ still holds in this case.
\end{Rmk}

\section{Complements for the multifractionnal Rosenblatt process} \label{sec:rosen}

In this last section, we take advantage of the expression \eqref{expansionchaos2} to improve Theorem \ref{thm:dim} in the case $d=2$, where the multifractional Hermite process is the multifractional Rosenblatt process. Let us start by introducing the notions of Malliavin calculus that we are going to use. Details can be read in the fundamental books \cite{MR2962301,MR2200233}.

Generally speaking, let $\mathcal{H}$ be a real separable Hilbert space with inner product $\langle \cdot, \cdot \rangle_{\mathcal{H}} $ and associated norm $\| \cdot \|_{\mathcal{H}}$. We call isonormal Gaussian process over $\mathcal{H}$ any centred Gaussian family $X= \{ X(f) \, : \, f \in \mathcal{H} \}$ defined on a probability space $(\Omega,\mathcal{F},\mathbb{P})$ and such that, for every $f,g \in \mathcal{H}$, $\mathbb{E}[X(f)X(g)] = \langle f, g \rangle_{\mathcal{H}}$. One can assume that $\mathcal{F}$ is the $\sigma$-field generated by $X$. For all $m \geq 1$, $\mathcal{H}^{\odot_m}$ is the $m$th symmetric  tensor product of $\mathcal{H}$ and $L^2(\Omega,\mathcal{H}^{\odot_m})$ is the class of $\mathcal{H}^{\odot_m}$-valuated random elements $F$ which are $\mathcal{F}$-measurable and such that $\mathbb{E} [\| F\|_{\mathcal{H}^{\odot_m}}^2 ] < \infty$. Let $\mathcal{S}$ be the set of all cylindrical random variables of the form
\begin{equation}\label{intro:defmallia}
F=g(X(f_1),\ldots,X(f_n))
\end{equation}
with $n \geq 1$, $f_j \in \mathcal{H}$ and $g$ infinitely differentiable such that all its partial derivatives have polynomial growth. If $F \in \mathcal{S}$ is of the form \eqref{intro:defmallia}, the $m$th Malliavin derivative of $F$ is the element of $L^2(\Omega,\mathcal{H}^{\odot_m})$ defined by
\[ D^m F = \sum_{j_1,\ldots,j_m=1}^n \frac{\partial^m g}{\partial x_{j_1} \ldots \partial x_{j_m}}(X(f_1), \ldots, X(f_n)) f_{j_1} \otimes \cdots \otimes f_{j_m}.\]
For all $m \geq 1$ and $p \geq 1$, $\mathbb{D}^{m,p}$ denote the closure of $\mathcal{S}$ with respect to the norm
\begin{equation}\label{eqn:defofthenorms}
\| \cdot \|_{m,p} \, : \, F \mapsto \left(\mathbb{E}[|F|^p] + \sum_{j=1}^m \mathbb{E}[ \|D^j F \|^p_{\mathcal{H}^{\otimes_j}}] \right)^\frac{1}{p}.
\end{equation}
For all $p \geq 1$,  $\mathbb{D}^{\infty,p}= \bigcap_{m \geq 1} \mathbb{D}^{m,p}$

In the sequel, we will heavily use the following fact which is contained in \cite[Theorem 3.1]{MR3132731}.

\begin{Lemma}\label{lemma:nualart1}
If $F \in \mathbb{D}^{2,s}$ is such that $\mathbb{E}[|F|^{2p}] < \infty$ and $\mathbb{E}[ \| DF\|_{\mathcal{H}}^{-2r}] < \infty$ for $p,r,s>1$ satisfying $\frac{1}{p}+\frac{1}{r}+ \frac{1}{s}=1$, then $F$ has continuous and bounded density $f_F$ with
\[ \sup_{x \in \R} |f_F(x)| \leq c_p \left\|\| DF\|_{\mathcal{H}}^{-2} \right\|_{L^r(\Omega)} \|F\|_{2,s},\]
where $c_p >0$ is a deterministic constant only depending on $p$.
\end{Lemma}

In our context, we work with $\mathcal{H}=L^2(\R)$ and, for all $f \in L^2(\R)$, $X(f)=I_1(f)$ is the Wiener-Itô integral of $f$ with respect to the Brownian motion. For all $p,d \geq 1$, $I_d(f) \in D^{\infty,p}$ and, for all $q \geq 1$,
\[
D^q I_d(f) = \begin{cases}
\frac{d!}{(d-q)!} I_{d-q}(f)& \text{if $q \leq d$} \\
0 & \text{otherwise,}
\end{cases}
\]
where, in $I_{d-q}(f)$, the stochastic integral is taken with respect to $d-q$ variables, resulting in a random variable belonging to $L^2(\Omega,L^2(\R^q))$. In particular, $D^d I_d(f) = d! f$. In the case $d=2$, we can use the expansion \eqref{expansionchaos2} and write
\[ D I_2(f) = 2 \sum_{j \in \N} \lambda_{f,j} I_1(e_{f,j}) e_{f,j}\]
which entails, given the orthogonality of the system $\{e_{f,j} \}$,
\begin{equation}\label{eqn:derivordre}
\| D I_2(f)\|_{L^2(\R)} = 2 \left(\sum_{j \in \N} \lambda_{f,j}^2 I_1(e_{f,j})^2 \right)^\frac{1}{2}.
\end{equation}
In particular, as $\{I_1(e_{f,j}\}_j$ are i.i.d. $\mathcal{N}(0,1)$ random variables, we deduce from \eqref{eq:norenvp} that
\begin{equation}\label{eq:derivordre12}
\mathbb{E}\left[\| D I_2(f)\|_{L^2(\R)}^2 \right] = 4 \left( \sum_{j \in \N} \lambda_{f,j}^2 \right)= 4 \|f\|^2.
\end{equation}

Now, let us state \cite[Lemma 7.1]{MR3132731} which gives an estimate for the negative moments of random variables of the form \eqref{eqn:derivordre}, which is particularly useful to apply Lemma \ref{lemma:nualart1}.

\begin{Lemma}\label{lemma:nualart2}
Let $G := \left( \sum_{j \in \N} \lambda_j X_j^2 \right)^\frac{1}{2}$ where $\{\lambda_j \}_{j \in \N}$ satisfies $|\lambda_j| \geq |\lambda_{j+1}|$ for all $j \geq 1$ and $\{X_j \}_{j \in \N}$ are i.i.d. standard normal. For all $r >1$, $\mathbb{E}[G^{-2r}] < \infty$ if and only if there exists $N > 2r$ such that $|\lambda_N| >0$ and, in this case,
\begin{equation}
\mathbb{E}[G^{-2r}] \leq c_p N^{-r} |\lambda|^{-2r},
\end{equation}
with $c_r >0$ a deterministic constant only depending on $r$.
\end{Lemma}

Let use Lemma \ref{lemma:nualart1} to improve Proposition \ref{lemma:boundproba} in the second order Wiener chaos.

\begin{Prop}\label{lemma:boundprobaorder2}
Given $d \in \N^*$, a compact set $K$ of $(\frac{1}{2},1)$ and a Hurst function $H \, : \, \R_+ \to K$ satisfying Condition \ref{condi} (a) and a compact interval $I \subset \R_+$. If $t_0 \in I$ is such that $H(t_0)=\underline{H}(I)$, there exist two deterministic constants $c>0,\xi>0$, both only depending on $d$, $H$ and $I$, such that, for all $0 < \varepsilon < \xi$, $x \geq 0$ and $t,u \in I \cap [t_0- \varepsilon,t_0+\varepsilon]$,
\[ \mathbb{P}(|X_d^{H(\cdot)} (t)-X_d^{H(\cdot)} (u)| \leq x) \leq c x  |t-u|^{-\overline{H}(t_0,\varepsilon)}. \]
\end{Prop}

\begin{proof}
Let us fix $t,u\in I$, we assume without loss of generality $u<t$. We keep the notation 
\[ f_{t,u}^{H(\cdot)} \, : \, \R^2 \to \R \, : \, \mathbf{w} \mapsto \int_0^t f_{H(t)}(s,\mathbf{w}) \, ds - \int_0^u f_{H(u)}(s,\mathbf{w}) \, ds   \]
introduced in the proof of Proposition \ref{lemma:boundproba} and also consider the function
\[ f_{t,u}^{H(t)} \, : \, \R^2 \to \R \, : \, \mathbf{w} \mapsto \int_u^t f_{H(t)}(s,\mathbf{w}) \, ds .  \]
Note that $I_2 \left( f_{t,u}^{H(t)}\right)= X_d(t,H(t))-X_d(u,H(t))$. If $\{\lambda_j\}_{j \in \N}$ are the eigenvalues of the Hilbert-Schmidt operator $\mathcal{A}_{f_{1,0}^{H(t)}}$ ordered with $|\lambda_j| \geq |\lambda_{j+1}|$, on one hand, one can check, with some obvious changes of variables, that  $\{|t-u|^{H(t)} \lambda_j\}_{j \in \N}$ are the eigenvalues of $\mathcal{A}_{f_{t,u}^{H(t)}}$. On the other hand, we know from the proof of \cite[Theorem 3.1]{MR2768856} that $\lambda_3 \neq 0$. Finally, inequality \eqref{eq:norenvp} allows to affirm that, if $\{\xi_j^{t,u}\}_{j \in \N}$ are the eigenvalues of the Hilbert-Schmidt operator $\mathcal{A}_{f_{t,u}^{H(u)}}$ ordered with $|\xi_j^{t,u}| \geq |\xi_{j+1}^{t,u}|$
\[ |\xi_3^{t,u}|> |t-u|^{H(t)} |\lambda_3|- \|  f_{t,u}^{H(\cdot)}-f_{t,u}^{H(t)}\|_{L^2(\R^2)}. \]
From Proposition \ref{prop:holdofgene}, we know that there exists $c_2$, only depending on $K$ and $H$, such that
\[ \|  f_{t,u}^{H(\cdot)}-f_{t,u}^{H(t)}\|_{L^2(\R^2)} \leq c_2 |H(t)-H(u)|.\]
Now, from the Condition \ref{condi} (a) for $H$, one can conclude, just as in the proof of Proposition \ref{lemma:boundproba}, that there exists $\xi>0$, a deterministic constant, only depending on $H$ and $I$, such that, for all $0 < \varepsilon < \xi$ and $t,u \in I \cap [t_0- \varepsilon,t_0+\varepsilon]$,
\[ |\xi_3^{t,u}|> \frac{|\lambda_3|}{2} |t-u|^{\overline{H}(t_0,\varepsilon)}. \]
It follows from Lemma \ref{lemma:nualart2}, that, for all such $\varepsilon$ and $t,u$ and, for all $r \in (1, \frac{3}{2})$, that
\begin{align*}
\left\|\| D \left(X_d^{H(\cdot)}(t)-X_d^{H(\cdot)}(u) \right)\|_{L^2(\R)}^{-2} \right\|_{L^r(\Omega)} \leq c_r \frac{4}{3} |\lambda_3|^{-2} |t-u|^{-2 \overline{H}(t_0,\varepsilon)},
\end{align*}
with $c_r >0$ a deterministic constant depending only on $r$. For all $p>1$, $\mathbb{E}\left[ \left(X_d^{H(\cdot)}(t)-X_d^{H(\cdot)}(u) \right)^p \right] < \infty$. Finally, from \eqref{eqn:defofthenorms}, equality \eqref{eq:derivordre12}, Proposition \ref{prop:holdofgene} and Condition \ref{condi} (a), we deduce the existence of a deterministic constant $c_1 >0$, only depending on $H$ and $I$, such that, for all $0 < \varepsilon < \xi$ and $t,u \in I \cap [t_0- \varepsilon,t_0+\varepsilon]$
\[ \left\|\left(X_d^{H(\cdot)}(t)-X_d^{H(\cdot)}(u) \right) \right\|_{2,2} \leq c_1 |t-u|^{\overline{H}(t_0,\varepsilon)}. \]
Since, as a consequence of the hypercontractivity property on the Ornstein-Uhlenbeck semi group \cite[Theorem 2.7.2]{MR2962301}, all the $\| \cdot \|_{m,p} $ norms are equivalent in any finite sum of Wiener chaoses, one can conclude, by Lemma \ref{lemma:nualart1}, the existence of two deterministic constants $c,\xi>0$, both only depending on $H$ and $I$, such that, for all  $0 < \varepsilon < \xi$ and $t,u \in I \cap [t_0- \varepsilon,t_0+\varepsilon]$, $\left(X_d^{H(\cdot)}(t)-X_d^{H(\cdot)}(u) \right)$ has a continuous density bounded by $c |t-u|^{-\overline{H}(t_0,\varepsilon)}$. The conclusion follows immediately
\end{proof}

The proof of Theorem \ref{thm:rose} is then a direct adoption of the one of Theorem \ref{thm:dim}, using the improved estimate given by Proposition \ref{lemma:boundprobaorder2}.

\begin{proof}[Proof of Theorem \ref{thm:rose}]
It suffices to repeat the proof of Theorem \ref{thm:dim} using Proposition \ref{lemma:boundprobaorder2} instead of Proposition \ref{lemma:boundproba}. Therefore, we remove the factor $\frac{1}{d}$ in the computations \eqref{estimationespedim} and get
\[\mathbb{E}  \left[ \left(|X_d^{H(\cdot)}(t+r)-X_d^{H(\cdot)}(t)|^2+r^2 \right)^{- \frac{s}{2}} \right] \leq c'' r^{1-s-\overline{H}(t_0,j^{-1})}. \]
\end{proof}

\begin{Rmk}
As previously, we can note that the proof for the lower bound for the Hausdorff dimension requires a weaker assumption for the Hurst function $H$, see Remark \ref{rmk:redcuhyp} here over.
\end{Rmk}

\bibliography{biblio}{}
\bibliographystyle{plain}

\end{document}